\documentclass[journal,twoside]{IEEEtran}
\usepackage{amsmath, amssymb, amsfonts}
\usepackage{silence}
\usepackage{graphicx, enumitem, xcolor, url, booktabs, multirow}
\usepackage[caption=false,font=footnotesize]{subfig}
\usepackage{microtype}
\usepackage{cite}
\usepackage[hidelinks]{hyperref}
\usepackage[capitalise,nameinlink]{cleveref}

\newenvironment{proof}{\begin{IEEEproof}}{\end{IEEEproof}}
\newtheorem{theorem}{Theorem}
\newtheorem{proposition}{Proposition}
\newtheorem{lemma}{Lemma}
\newtheorem{corollary}{Corollary}
\newtheorem{definition}{Definition}
\newtheorem{remark}{Remark}
\crefname{proposition}{Proposition}{Propositions}
\crefname{remark}{Remark}{Remarks}
\crefname{definition}{Definition}{Definitions}
\crefname{subsection}{Subsection}{Subsections}
\crefname{section}{Section}{Sections}

\newif\ifshowrev
\showrevtrue
\definecolor{revblue}{RGB}{0, 70, 170}

\definecolor{NCSUred}{RGB}{153, 0, 0}
\definecolor{NCSUgreen}{RGB}{0, 132, 115}
\definecolor{NCSUblue}{RGB}{65, 86, 161}
\definecolor{NCSUorange}{RGB}{209, 73, 5}

\newcommand{\mc}[1]{\mathcal{#1}}
\newcommand{\ms}[1]{\mathsf{#1}}
\newcommand{\mr}[1]{\mathrm{#1}}

\newcommand{\mbb}[1]{\mathbb{#1}}
\newcommand{\rg}[1]{\mathring{#1}}
\newcommand{\mR}{\mathbb{R}}
\newcommand{\mC}{\mathbb{C}}
\newcommand{\mN}{\mathbb{N}}
\newcommand{\xD}[1]{\mr{d} #1}
\newcommand{\xDD}[2]{\frac{\xD{#1}}{\xD{#2}}}

\newcommand{\bra}[1]{\left( #1 \right)}
\newcommand{\Bra}[1]{\left[ #1 \right]}
\newcommand{\BRA}[1]{\left\{ #1 \right\}}

\newcommand{\ip}[2]{\left\langle #1, \, #2 \right\rangle}

\newcommand{\ve}{\varepsilon}
\renewcommand{\ker}{\mr{ker}\,}
\newcommand{\ran}{\mr{ran}\,}
\renewcommand{\Re}{\mr{Re}\,}

\begin{document}

\title{Koopman Observability of Nonlinear Systems and Data-Driven Koopman--Luenberger Observer}

\author{Wentao Tang~\IEEEmembership{Member, IEEE}
\thanks{This work was supported by the National Science Foundation under Award No.~2414369 and CAREER Award No.~2543029.}%
\thanks{W. Tang is with the Department of Chemical and Biomolecular Engineering, North Carolina State University, Raleigh, NC 27695, USA (e-mail: wtang23@ncsu.edu).}}

\markboth{\LaTeX ~Template}
{Tang: Koopman Observability and Koopman--Luenberger Observer}

\maketitle

\begin{abstract}
While nonlinear systems can be generally described as infinite-dimensional linear systems via their Koopman semigroups, the \emph{direct synthesis of state observers} on the premise of learned Koopman operator models, in an analogous manner to finite-dimensional linear systems, remains an open problem.
In this paper, the concepts of \emph{Koopman observability}, \emph{Koopman detectability}, and \emph{Koopman--Luenberger observer} of nonlinear systems are provided.
By defining the Koopman semigroup on a reproducing kernel Hilbert space (RKHS), the conditions for Koopman exponential detectability are established, which involves solving an \emph{operator Lyapunov inequality} that only implicitly involves the infinitesimal generator of the Koopman semigroup. 
Through the use of a finite-horizon Yosida approximation of the Koopman generator and an empirical sampling operator, the data-driven solution of the observer gain can be reduced to a finite-dimensional subspace specified by the sample and hence fulfilled by convex semidefinite programming (SDP). 
The state observation error is guaranteed to be bounded in a mean-squares sense, under closed-loop dissipation conditions on the sampling operator and its empirical operator.
\end{abstract}

\begin{IEEEkeywords}
Koopman operator, observability, observer synthesis, nonlinear systems, reproducing kernel Hilbert space.
\end{IEEEkeywords}

\section{Introduction}\label{sec:intro}
The operator-theoretic approaches for nonlinear systems have gained attention of researchers in the recent decade \cite{mauroy2020koopman, mezic2021koopman, brunton2022modern}. 
Using the Koopman operator (in discrete time) or Koopman semigroup (in continuous time), a nonlinear system can be viewed as an infinite-dimensional linear systems on a function space. 
The learning of Koopman operator (semigroup) from data, originated from dynamic mode decomposition methods \cite{bagheri2013koopman, williams2015data} and more recently analyzed in depth through statistical learning theory \cite{klus2020eigendecompositions, kostic2022learning, nuske2023finite}, shows a promise of operator-based data-driven control \cite{tang2022data}. 
Intuitively, one would hope that, by directly using the learned operator models, the control-theoretic problems for nonlinear systems and their solutions can \emph{resemble and generalize} the ones for linear systems in finite dimensions; such an idea \cite{korda2018linear}, however, does not na{\"{i}}vely come true without using formal analysis. 

\par While extensive practical studies aimed to connect Koopman operator learning to nonlinear controllers and observers \cite{otto2021koopman} and apply Koopman-based controllers \cite{narasingam2019koopman, bruder2020data}, arguably, a rigorous foundation of Koopman-based control and observation remains incomplete \cite{haseli2025two, strasser2026overview}. 
So far, \emph{direct synthesis} formulations for observers and controllers based on Koopman operators are essentially absent. This is largely due to the following three major issues. 
\begin{enumerate}[label=(\roman*)]
    \item \textit{Bilinearity due to inputs.} For a controlled system, even an input-affine one ($\dot{x} = f(x)+G(x)u$), the operator representation of the model must be at least \emph{bilinear} \cite{goswami2021bilinearization, iacob2024koopman, morris-ye-tang2026}. 
    Hence, control methods that assume fully linear lifted models \cite{proctor2018generalizing} tend to lack accuracy and stability guarantees. 
    Under bilinearity, on the other hand, direct synthesis is much harder and often results in weaker stability guarantees, suboptimal design, or still computationally expensive MPC \cite{strasser2024koopman, bold2025kernel}. 
    In several works on Koopman-based optimal control \cite{huang2022convex, bevanda2024kernel}, the bilinear term combining the inputs and occupation measure (or probability densities) was considered holistically in dual (measure-based) formulations, which is effective for stochastic systems with non-vanishing densities. 
    \item \textit{Violation of invariance.} Despite the easiness of defining the Koopman operator on a finite dictionary, the operator (semigroup) tends not to be invariant in finite dimensions \cite{haseli2026modeling}. 
    Although single-step prediction errors can be bounded statistically \cite{korda2018convergence, bevanda2024nonparametric}, the effect of the lack of invariance on control or observation is not easy to characterize. 
    A more rigorous approach is to adopt a \emph{reproducing kernel Hilbert space (RKHS) that is equivalent to a Sobolev space} and thus invariant under sufficiently smooth dynamics \cite{hou2023sparse, bevanda2024nonparametric, kohne2025error}. 
    When the approximation on the RKHS has a uniformly bounded error, the closed-loop stability under \emph{given controllers} can be certified by carrying the errors in the model-based proofs \cite{bold2025kernel} and \emph{policy evaluation} (i.e., assessing performance under the given controller) can be approximated \cite{breiten2023approximability, breiten2025unifying, tang2025koopman}. 
    \item \textit{Lack of operator-based stability criteria.} A fundamental difficulty of direct operator-based synthesis lies in the lack of \emph{spectrum--stability relation}, which is the cornerstone of classical linear systems and control theory.\footnote{More explicitly, a Koopman operator (semigroup) with spectrum on the unit disk (left half plane) may not be related to stability. For example, if using a radial kernel that specifies a Sobolev-type RKHS (as in \cite{strasser2025kernel, bold2025kernel}; see \Cref{sec:prelim}), the spectral radius of the Koopman operator is always at least $1$, regardless of stability. }  
    The connection between stability and the Koopman eigenfunctions, was discussed in \cite{mauroy2016global}. For holomorphic systems on the complex plane ($\mC$) or $\mC^d$, the Bargmann function space can be used to define the Koopman operator \cite{mezic2020spectrum}. 
    In \cite{tang-ye2025koopman}, a linear--radial product kernel is proposed to define a RKHS, on which the Koopman spectrum--stability relation is found to be analogous to linear systems (see \Cref{subsec:linear--radial}). This construction allows the direct estimation of Lyapunov functions \cite{tang-ye2026koopman} and dissipativity \cite{tang-ye2026dissipativity} based on Koopman operator learning using linear operator inequalities, as well as the estimation of Zubov function that characterizes the domain of attraction \cite{tang2026data}. 
\end{enumerate}

\par This paper focuses on constructing a \emph{Koopman theory and method of state observation} for continuous-time autonomous nonlinear systems.\footnote{The controller synthesis problem is not attempted in the present paper. The author aims to address this issue, focusing on resolving the bilinearity due to the existence of inputs, in a collateral work \cite{tang2026koopman}.}  
Specifically, in the present work, \emph{Koopman observability and detectability} are defined in a similar way to those in infinite-dimensional linear systems theory \cite{curtain2020introduction}, but accounting for the fact that the observation target is only the finite-dimensional state instead of its ``lifted'' representation in the infinite-dimensional RKHS. 
Under suitable conditions, a \emph{Koopman--Luenberger observer}, specified by its observer gain as a bounded linear operator on the RKHS, is expected to be obtained by solving an operator Lyapunov inequality, whose numerical solution is to be approximated from a data-based estimate of the infinitesimal generator of the Koopman semigroup. 
As such, a convex optimization-based approach is provided for the synthesis of a fully linear Luenberger-type observer, which is ideally infinite-order and realistically approximated by data as a finite-order one.

\par Indeed, in the recent literature, a considerable amount of studies have used data-driven approaches, especially neural network-based approximators, to obtain nonlinear state observers \cite{peralez2021deep, alvarez2024nonlinear, tang2024synthesis, woelk2026neural}. These works utilize the Kazantzis--Kravaris/Luenberger observer structure \cite{kazantzis1998nonlinear, andrieu2006existence, bernard2018luenberger}, where the states are injectively mapped into some observer states in a higher but still finite-dimensional space and the pseudo-inverse of such an injection can be estimated from data. 
It was then understood that the use of Koopman eigenfunctions or Koopman spectrum, which covers the evolution in infinite dimensions, can be helpful to efficiently estimate the injection mapping therein and the left pseudo-inverse of this injection (which returns the state estimates) \cite{ni2026, tang2025data}. On the other hand, the estimation of the injection's pseudo-inverse tends to be cumbersome. 
This paper, therefore, considers generic Luenberger-type observers, where \emph{the injection becomes infinite-dimensional}, valued in an RKHS, and the retrieval of state from the injection is linear. In a previous work \cite{ye2026}, a Koopman-based robust observer synthesis approach was proposed, where the Koopman model was approximated by extended dynamic mode decomposition (EDMD) with a monomial dictionary and hence, but guaranteed to have a sectorially bounded error that allows direct synthesis of a finite-order observer. Here, instead of choosing a finite-dimensional dictionary (feature), a kernel method is used to directly allow an infinite-dimensional feature. 

\paragraph{Related work}
The idea of such a general Koopman-type observer originated in Surana et al. \cite{surana2016linear} and was also used in Mesbahi et al. \cite{mesbahi2021nonlinear}, assuming that a finite number of Koopman eigenfunctions can be accurately estimated and designing a Luenberger observer in their span.\footnote{
	Such an eigenfunction approach is, however, theoretically vulnerable. When the Koopman spectrum is not discrete, eigenfunctions do not even exist. The assumption that the output mappings and the state component mappings $e_k: x\mapsto x_k$ exactly lie in the span of these eigenfunctions, is typically unverifiable if not false at all. 
}
More recently, Mohet et al. \cite{mohet2025dual} formulated the \emph{dual Koopman system} on an RKHS, i.e., the linear system governed by the adjoint Koopman generator, and the associated Luenberger observer (the \emph{dual Koopman observer}). When the generator admits a spectral decomposition via biorthonormal Koopman eigenfunctions, if the decomposition is learnable, the observer can be synthesized by separating the spectrum at $\Re\lambda=\beta$ and performing a finite-dimensional pole placement to guarantee exponential error convergence; see also \cite{mohet2026density} for an extension to density-based initial conditions. 
The present paper also adopts this dual viewpoint, but takes a complementary route: (a) the RKHS is a Sobolev-type space with a linear--radial kernel, which does not require holomorphic dynamics; (b) only the state readout $x=Ez$ is required to converge, leading to an operator Lyapunov inequality instead of spectral assignment, which does not require eigenvalue--eigenfunction expansion; and (c) this inequality is solved from data by convex semidefinite programming (SDP), with learning error analysis, instead of assuming exactly learned spectrum of an unbounded generator.\footnote{
	As an ``expense'' of such a more computationally direct and flexible approach, we decide not to guarantee exponential decay rate (see \Cref{rem:spectral-gap}). In fact, when the system is oscillatory or chaotic, exponential decay on the infinite-dimensional RKHS is anyways impossible. In other words, in a Koopman--RKHS framework, exponential convergence is only possible for systems that are known to have finite unstable poles or already stable. In this sense, such a theoretical pursuit hinders the capability to handle a larger range of systems. 
}

\paragraph{Contributions of this work}
The technical contributions of the present paper are summarized as follows. 
\begin{enumerate}[label=(\roman*)]
    \item Notions of \emph{Koopman observability and detectability} tailored to the state readout are defined for the triplet $(A,C,E)$ on a properly chosen RKHS (\Cref{subsec:KO}, \Cref{subsec:KOS}). 
    The explicit conditions in the case where the Koopman semigroup has a discrete spectrum are given (\Cref{subsec:KO.test}).
    \item Explicit conditions for the existence and synthesis of an infinite-dimensional observer gain (operator) are provided, boiling down to an \emph{operator Lyapunov inequality} on the RKHS that involves the \emph{infinitesimal generator of the Koopman semigroup} (\Cref{subsec:KD}). 
    \item The Koopman generator, as an unbounded operator, is related to two integral operators, and via Yosida approximation, the operator Lyapunov inequality is converted into a version with only compact operators. With finite data trajectories, the \emph{empirical version of the Lyapunov inequality} is provided (\Cref{subsec:compactified}), and the effect of learning error on observer performance is analyzed (\Cref{subsec:learning.error}). 
    \item For practical implementation, the infinite-order observer must be approximated by a finite order one. It is proved that the finite-order observer using the solution of the empirical Lyapunov inequality leads to a state observation error that converges to zero asymptotically (\Cref{subsec:performance}). The elementary expressions in data-based finite-dimensional space for the operator computations are given (\Cref{subsec:elementary}). 
    \item The Koopman--Luenberger observer synthesis is applied to different cases in \Cref{sec:numerical}, confirming its performance despite complicated state-space geometries and its advantage over neural network training of KKL observers. 
\end{enumerate}
Technically, our Koopman observability and observer theory is built on the existing semigroup theory of infinite-dimensional linear systems \cite{li1995optimal, curtain2020introduction}, mainly intended for studying linear partial differential equations and retarded equations. The use of this semigroup-theoretic language for Koopman-based state estimation was pioneered in \cite{mohet2025dual}; the difference with the present work was discussed earlier.

\paragraph{Notation} We use lower-case Latin and Greek letters for scalar and vector-valued objects, and upper-case letters for matrices and operators. Matrix approximations of operators will use Sans Serif fonts (e.g., $\ms{P}$ for $P$). $\mN=\{1, 2, \cdots\}$, $\mR_+=[0, \infty)$, $\mR^d$, and regions or intervals thereon use blackboard fonts. The component functions $x\mapsto x_k$ are denoted as $e_k$, while the standard basis vectors of $\mR^d$ are denoted as $\ms e_k$ ($k=1,\cdots,d$) in Sans Serif for the sake of distinction. 
Function spaces are denoted by normal ($H^s(\mbb{X})$ --- Sobolev space, namely $W^{s,2}(\mbb X)$) or calligraphic upper-case letters (e.g., $\mc{H}_\kappa$ --- RKHS). Operator spaces use calligraphic letters (e.g., $\mc{B}$ --- space of bounded operators). Norms are denoted by $\|\cdot\|$ or $|\cdot|$ (in $\mR^d$), inner products are $\ip{\cdot}{\cdot}$, and outer product between two vectors on Hilbert space (which forms a rank-$1$ operator) is denoted by $\times$, i.e., $(z_1\times z_2)z_3 = \ip{z_2}{z_3}z_1$ (instead of $\otimes$ to avoid confusion with the Kronecker product which makes $\ip{z_1\otimes z_2}{z_3\otimes z_4} = \ip{z_1}{z_3}\ip{z_2}{z_4}$). A positive (semi-)definite matrix or positive (nonnegative) operator is indicated by $P\succ(\succeq)0$, and hence $P\succeq Q$ indicates $P-Q\succeq 0$ (and similarly for $\succ$, $\preceq$, and $\prec$). For $a, b\geq0$, we write $a\lesssim (\propto) b$ to mean that $a\leq(=) \mr{const}\cdot b$.

\section{Preliminaries on Semigroups and RKHS}\label{sec:prelim}

We consider a nonlinear system 
\begin{equation}\label{eq:sys}
    \dot{x}(t) = f\bra{x(t)}, \enspace y(t) = h\bra{x(t)} 
\end{equation}
whose state space is a subset $\mbb{X}$ of $\mR^d$. We assume that the flow of $f$ is complete in $\mbb{X}$, i.e., for all $t\in \mR_+$, the solution map $S^t$ that maps $x(s)$ to $x(s+t)$ on any orbit $x(\cdot)$ for any $s\in \mR$ maps $\mbb{X}$ into $\mbb{X}$. Thus, $\{S^t\}_{t\in \mR_+}$ is a one-parameter semigroup. We consider $h: \mbb{X}\rightarrow \mR^m$ to have sufficient regularity (to be specified later). 
Using operator-theoretic notions, the system \eqref{eq:sys} can be formally converted to an infinite-dimensional linear system. Such a treatment shall be useful for observer synthesis. 

\subsection{Koopman Semigroup Theory}
\begin{definition}
    The \emph{Koopman semigroup} of system \eqref{eq:sys}, $\{K^t\}_{t\in \mR_+}$, defined on a space $\mc{G}$ of state-dependent functions\footnote{Such a state-dependent function $g\in \mc G$ is often called an ``observable'' in the literature. The author tends to avoid this terminology in control theory as it may hamper the discussion on state observation and observability problems.}, is given by $K^t g = g\circ S^t$ for all $g\in \mc{G}$. (For well-definedness, $\mc{G}$ must be closed under the action of $K^t$ for all $t\in \mR_+$.) 
    If $\mc{G}$ is a Hilbert space (with inner product $\ip{\cdot}{\cdot}_{\mc{G}})$, then the adjoint Koopman semigroup $\{T^t\}$ ($T^t = K^{t*})$, is called the \emph{Perron--Frobenius semigroup}. In other words, $\ip{T^tg_1}{g_2}_{\mc{G}} = \ip{g_1}{K^tg_2}_{\mc{G}}$ for all $g_1,\,g_2\in \mc{G}$.
\end{definition}

The choice of function class $\mc{G}$ obviously depends on the regularity of $f$. We are in particular interested in Hilbert spaces, as infinite-dimensional linear systems theory typically adopts a Hilbert space formulation \cite{curtain2020introduction}. 
Here it is justified that the Sobolev--Hilbert space $H^s(\mbb X)$ can be used.\footnote{The idea was first proposed in \cite{kohne2025error} for discrete-time systems, where the boundedness of Koopman operator was proved. The extension to continuous-time system and the proof here for the strong continuity of Koopman semigroup is natural.} 
Specifically, 
\begin{equation*}
    H^s(\mbb X) = \Big\{ g: \mbb{X}\to\mR: \|g\|_{H^s}^2 := \sum_{|\alpha|\leq s} \int_{\mbb X} |\partial^\alpha g|^2 < \infty\Big\}, 
\end{equation*}
where $\partial^\alpha$ denotes the generalized derivative with respect to the multi-index $\alpha\in\mN_0^d$.
\begin{proposition}\label{prop:invariance.1}
    If $f\in C_\mr{b}^s(\mbb{X})^d$ for some $s\in\mbb{N}$ (i.e., all components of $f$ and their derivatives up to $s$-th order are bounded in $\mbb{X}$), then $\{K^t\}$ is a strongly continuous semigroup on $\mc{G} = H^{s'}(\mbb{X})$ for all $0\leq s'\leq s$. 
\end{proposition}
\begin{proof}
    We only need to prove for $s'=s$. For any $g\in H^s(\mbb{X})$, the generalized derivative $\partial^\alpha (K^tg)(x) = \partial^\alpha (g\circ S^t)(x)$ with respect to any multi-index $\alpha$ of length $r\leq s$, by the chain rule, contains terms in the form of a product of $\partial^\beta g(S^t(x))$ and partial derivatives of $S^t$, where $\beta\leq \alpha$. In particular, the term with $\beta=\alpha$ is followed by $\partial_{\alpha_1} S^t(x)\cdots \partial_{\alpha_d} S^t(x)$, while every other term has a factor of higher-order derivatives of $S^t$. 
    This is known as the multivariate version of the Fa{\`{a}} di Bruno's formula. Given the conditions, $S^t \in C^s(\mbb{X})$ and at small $t$, $S^t(x)=x+O(t)$. 
    Therefore, $\partial^\alpha (K^tg)\rightarrow \partial^\alpha g$ pointwise, and hence in $L^2$ due to Lebesgue convergence theorem, as $t\downarrow 0$. That is, $\lim_{t\downarrow 0} K^tg = g$ in $H^s(\mbb{X})$ for all $g\in H^s(\mbb{X})$. 
\end{proof}

The $H^s$ space, however, does not intrinsically contain any construction pertaining to the \emph{equilibrium point}. In many applications (especially if the observer synthesis problem is associated with control), it is known \textit{a priori} that the origin $0\in \mR^d$ is an equilibrium point (i.e., $f(0)=0$, usually after translating the state coordinates), then it is convenient to revise the function spaces, with the purpose of restricting to functions that take zero values at the origin.
\begin{definition}
    For any function space $\mc{G}$, we denote the function space with linear factors:
    \begin{equation}
        \rg{\mc{G}} = \BRA{ \sum_{k=1}^d e_kg_k: g_1, \cdots, g_d\in \mc{G} }, 
    \end{equation}
    where $e_k(x) = x_k$ ($k=1,\cdots,d$) are the component functions. 
\end{definition}
We are particularly interested in $\rg{H}^s(\mbb{X})$, which can be viewed as a rigorized infinite-dimensional version of the ``monomial dictionary'' used in extended dynamic mode decomposition \cite{williams2015data} or monomial-like basis functions in Koopman-based control works \cite{strasser2024koopman}. 
The following conclusion, extended from the discrete-time case in \cite{tang-ye2025koopman}, can be proved in a similar manner to the foregoing proposition (e.g., see \cite{ye2026}). 
\begin{proposition}\label{prop:invariance.2}
    If $f\in \rg{C}^s_{\mr{b}}(\mbb{X})^d$ for some $s\in\mN$, then $\{K^t\}$ is a strongly continuous semigroup on $\rg{H}^{s'}(\mbb{X})$, $\forall s'\in[0,s]$. 
\end{proposition}

Hence, if the system is regular enough to satisfy the condition of \Cref{prop:invariance.1} or \Cref{prop:invariance.2}, and the function space $\mc{G}$ is chosen in concert, then we have an infinite-dimensional linear system governed by the infinitesimal generator of the Koopman semigroup $\{K^t\}$ or its adjoint $\{T^t\}$. 
Naturally, we would like that the state $x(t)$ can always be boundedly extracted from the infinite-dimensional functional state. To this end, the construction of RKHS \cite{steinwart2008support} is needed.

\subsection{Representation on Reproducing Kernel Hilbert Spaces}
\begin{definition}
    Given $\kappa\in C(\mbb{X}\times \mbb{X}, \mR)$ such that for any $\{x^{(i)}\}_{i=1}^n \subset \mbb{X}$ with $n\in \mN$, $\Bra{ \kappa\bra{x^{(i)}, x^{(j)}} }_{i,j=1}^n \succeq 0$, the \emph{RKHS} with \emph{kernel} $\kappa$, denoted as $\mc{H}_\kappa(\mbb X)$ or simply $\mc H_\kappa$, is the closed span of all $\kappa(x, \cdot)$ with $x\in \mbb{X}$, endowed with inner product: $\ip{\kappa(x,\cdot)}{\kappa(x',\cdot)} = \kappa(x,x')$ and hence norm: $\|\kappa(x, \cdot)\| = \kappa(x,x)^{1/2}$. 
    The RKHS has the following \emph{reproducing property}: $\forall g\in \mc{H}_\kappa$, $\forall x\in \mbb{X}$, $\ip{g}{\kappa(x, \cdot)} = g(x)$. 
\end{definition} 

\par It is well-known that a Sobolev--Hilbert space with a large smoothness index $s>d/2$ coincides with an RKHS, with equivalent norms \cite{wendland2004scattered}. 
\begin{lemma}\label{lem:Sobolev-RKHS.1}
    Suppose that $\kappa$ is a \emph{radial kernel}, i.e., $\kappa(x, x') = \rho(|x-x'|)$ for some $\rho: \mR_+\rightarrow \mR_+$, and that the Fourier transform of $\rho$, $\hat{\rho}(\xi)$, satisfies $|\hat{\rho}(\xi)| \propto (1+|\xi|^2)^{-s}$ where $s>d/2$. In addition, assume that $\mbb{X}$ is a region with Lipschitz boundary. Then $H^s(\mbb{X}) \equiv \mc{H}_\kappa(\mbb X)$.\footnote{
    The radial function that satisfies the condition can be easily constructed by the inverse Fourier transfer function, e.g., 
    \begin{equation*}
    	\kappa(x,x') = \int_{\mR^d} \bra{1+|\xi|^2}^{-s} e^{-i\xi\cdot(x-x')}  \xD{\xi} \propto |x-x'|^\nu \mr{K}_\nu(|x-x'|) 
    \end{equation*}
    where $\nu = s-d/2>0$ and $\mr{K}$ is the modified Bessel function of the second kind. This is known as the Bessel potential kernel. }
\end{lemma}
 
\begin{proposition}\label{prop:invariance.3}
    Suppose that $\mbb{X}$ is a bounded region with Lipschitz boundary and $f\in C^s(\mbb{X})^d$ for some $s>d/2$. Then $\{K^t\}$ is strongly continuous semigroup on the RKHS $\mc{H}_\kappa(\mbb{X})$ such that $H^s(\mbb{X}) \equiv \mc{H}_\kappa(\mbb{X})$. 
\end{proposition}

In the above statement, the kernel of the Sobolev-type RKHS is a radial function, i.e., dependent only on $|x-x'|$. We again consider modifying the space into $\rg{H}^s(\mbb{X})$ whose members are functions containing linear factors $e_1, \cdots, e_d$. As argued in the author's previous work \cite{tang-ye2025koopman}, $\rg{H}^s(\mbb{X})$ is an RKHS with ``\emph{linear--radial kernel}'' $\rg{\kappa}$.\footnote{In fact, $\rg H^s(\mbb{X})$ is the tensor product of $\mr{span}\{e_k\}_{k=1}^d$ (which is trivially a RKHS) and $H^s(\mbb{X})$, and hence the kernel $\rg\kappa$ is naturally the product kernel formed by the linear kernel $(x, x')\mapsto x\cdot x'$ and the radial kernel $\kappa$ that underlies $H^s(\mbb{X})$.}
\begin{lemma}\label{lem:Sobolev-RKHS.2}
    If $H^s(\mbb{X})=\mc{H}_\kappa$, then $\rg{H}^s(\mbb{X}) \equiv \mc{H}_{\rg{\kappa}}(\mbb{X})$, where $\rg{\kappa}(x,x')= (x\cdot x')\kappa(x,x')$ for any $x,x'\in \mbb{X}$. 
\end{lemma}
\begin{proposition}\label{prop:invariance.4}
    Suppose that $\mbb{X}$ is a bounded region with Lipschitz boundary and $f\in \rg{C}^s(\mbb{X})$ for some $s>d/2$. Then $\{K^t\}$ is strongly continuous semigroup on the RKHS $\mc{H}_{\rg{\kappa}}$ such that $H^s(\mbb{X}) \equiv \mc{H}_\kappa$ (and thus $\rg H^s(\mbb{X}) \equiv \mc{H}_{\rg \kappa}$). 
\end{proposition}

\par Now under the conditions of \Cref{prop:invariance.1} and \Cref{prop:invariance.3}, or \Cref{prop:invariance.2} and \Cref{prop:invariance.4}, we can always map (actually inject) the any state $x\in \mbb{X}$ into a corresponding element on the RKHS, namely its canonical feature $\kappa(x, \cdot)$ or $\rg{\kappa}(x, \cdot)$. 
\begin{itemize}
    \item Due to the fact that 
    \begin{equation*}\ip{K^tg}{\kappa(x, \cdot)} = g(S^t(x)) =\ip{g}{\kappa(S^tx, \cdot)},\end{equation*} 
    we have $K^{t*} \kappa(x, \cdot) = T^t \kappa(x, \cdot) = \kappa(S^t(x), \cdot)$. 
    This implies that the evolution of $z(t)=\kappa(x(t), \cdot)$ on the RKHS, as an infinite-dimensional state space, is governed by the infinitesimal generator of the strongly continuous Perron--Frobenius semigroup $\{T^t\}$, which we will write as $A: \mc{D}(A)\rightarrow Z= \mc{H}_\kappa$, where $\mc D(A)$ is the domain of $A$ and a dense subspace of $Z$. 
    \item For the measured outputs $y_k=h_k(x)$ ($k=1,\cdots,m$), if $h_k\in \mc{H}_\kappa$ (or $\mc{H}_{\rg\kappa}$, we write by reproducing property that $y_k = \ip{h_k}{\kappa(x, \cdot)}$. Stacking the components together, we have 
    \begin{equation*}y = \begin{bmatrix} \ip{h_1}{\kappa(x, \cdot)} \\ \vdots \\ \ip{h_m}{\kappa(x, \cdot)} \end{bmatrix} = C\kappa(x, \cdot),\end{equation*}
    where $C$ is a bounded linear operator from $Z$ to $\mR^m$. 
    \item Importantly, for state observation, we are interested in $x_1, \cdots, x_d$. If $e_1, \cdots, e_d$ belong to the RKHS $Z$, then we have $x_k = e_k(x)=\ip{e_k}{\kappa(x, \cdot)}$ and therefore 
    \begin{equation*}x =\begin{bmatrix} \ip{e_1}{\kappa(x, \cdot)} \\ \vdots \\ \ip{e_d}{\kappa(x, \cdot)} \end{bmatrix} = E\kappa(x, \cdot)\end{equation*}
    with $E \in \mc B(Z, \mR^d)$.
    \item In the above bullet points, $\kappa$ can be replaced by $\rg\kappa$. We in particular remark that if $\rg\kappa$ is used as the kernel to define the RKHS $Z$, then $z = \rg\kappa(x,\cdot)$ and $x$ are effectively isometric: 
    \begin{equation}\label{eq:isometric}
        \|z\|^2 = \rg\kappa(x,x) = \rho(0)|x|^2 \propto |x|^2 = |Ez|^2. 
    \end{equation}
    Such an isometric relation will be beneficial for the theoretical analysis in the next section. 
\end{itemize}
Now we have the following representation of \eqref{eq:sys}: 
\begin{equation}\label{eq:ACE}
    \dot{z} = Az, \enspace y=Cz, \enspace x=Ez.
\end{equation}
The unbounded operator $A$ generates a strongly continuous semigroup $\{T^t\}$ on $Z$, $C\in \mc{B}(Z, \mR^m)$ is the \emph{measurement operator}, and $E\in \mc{B}(Z, \mR^d)$ is the \emph{state evaluation operator}. The first two equations of \eqref{eq:ACE} constitute the \emph{dual Koopman system} introduced by Mohet et al. \cite{mohet2025dual}. But here the evaluation operator $E$ is made explicit, since the state $x$ is retrieved from $z$ by a bounded linear map; this fact will be central to the observability notions below.

\subsection{Koopman Spectrum and Stability}\label{subsec:linear--radial}
It is worth slightly further remarking on the use of $\rg H^s$ instead of $H^s$. If considering $\{K^t\}$ and $\{T^t\}$ on $H^s$, then the equilibrium point at the origin results in $T^t\kappa(0, \cdot) = \kappa(0, \cdot)$ for all $t\in \mR_+$, where $\|\kappa(0, 0)\| = \rho(0)>0$. Hence for $\{T^t\}$, the \emph{spectrum of $A$}, its infinitesimal generator, cannot be confined completely on the open left half plane $\{\lambda\in\mC:\Re \lambda<0\}$, regardless whether the equilibrium point is attractive. 
Such a \emph{spectrum--stability relation} is unnatural and differs from the finite-dimensional case. 
The situation becomes different when using $\rg\kappa$. As $\rg\kappa(0, \cdot) = 0\kappa(0, \cdot) \equiv 0$, asymptotic convergence to the origin in $\mbb{X}\subseteq \mR^d$ comes along with the convergence to the origin in the Hilbert space $Z$. In the spirit of \cite{mezic2020spectrum}, by assuming the regularity of a homeomorphism, the following conclusion can be found \cite[Theorem 8]{tang-ye2025koopman}. 
\begin{proposition}\label{prop:stability-spectrum}
    Assume the conditions in \Cref{prop:invariance.4} and a homeomorphism $\psi\in \rg C^s(\mbb{X})^d$ such that for some $t\geq 0$, for all $x\in \mbb X$, $\psi\circ S^t\circ \psi^{-1}(x) = \mr{e}^{tF}x$, where $F = \mr{D}f(0)$ (the Jacobian matrix at the origin) is assumed to be Hurwitz. 
    Then $\sigma(A) \subset \{\lambda\in \mC: \operatorname{Re} \lambda \leq \max\{\operatorname{Re}\mu: \mu\in \sigma(F)\}\}$. 
\end{proposition} 
\begin{proof}
    Due to the fact that $e^{t\sigma(A)} \subseteq \sigma(T^t)$ for all $t\in \mR_+$, which was proved in Theorem 2.3 of Pazy \cite{pazy1983semigroups}, it suffices to prove the assertion that 
    \begin{equation*}
        \sigma(T^t) \subset \BRA{\lambda\in \mC: |\lambda|\leq \max\{|\mr{e}^{\mu t}|: \mu\in \sigma(F)\} }. 
    \end{equation*}
    To this end, take any $g\in \rg H^s(\mbb{X})$, hence $\tilde{g} = g\circ\psi^{-1} \in \rg H^s(\psi(\mbb X))$, and examine $K^{nt} g(x) = \tilde g(\mr{e}^{ntF}y)$ as a function of $y\in \psi(\mbb{X})$ at large $n\in \mbb{N}$. 
    Expressing $\tilde{g}$ as $\tilde{g} = \sum_{k=1}^d e_k\tilde{g}_k$ with $\tilde{g}_1, \cdots, \tilde{g}_d\in H^s(\psi(\mbb{X}))$, we have 
    \begin{equation*}
    	\|K^{nt}g\|_{\rg H^s}^2 \lesssim \sum_{\ell=1}^d \sum_{|\alpha|\leq s} \int_{\mbb X'} \bigg| \sum_{k=1}^d(\mr{e}^{ntF})_{k\ell} \partial_y^\alpha \bra{ \tilde{g}_k(\mr{e}^{ntF}y)} \bigg|^2 \xD{y}
    \end{equation*}
    with $\mbb X' :=\psi(\mbb X)$. 
    As $n$ increases, the evaluation of $\partial_y^\alpha$ on $\tilde{g}_k(\mr{e}^{ntF}y)$, when taking the squared integral, results in the multiplication of operator norm of $\mr{e}^{ntF}$. Denote by $\sigma_1(F)$ the eigenvalue of $F$ with largest real part. Then we have the dependence on $n$ bounded by $\|K^{nt}g\|_{\rg H^s}\lesssim \mr{e}^{nt\cdot \operatorname{Re}\sigma_1(F)}$ and hence the spectral radius of $K^t$ bounded by $t\cdot \operatorname{Re}\sigma_1(F)$. Thus, 
    \begin{equation*}
        T(K^t) = \BRA{ \lambda\in \mC: |\lambda|\leq \max\{|\mr{e}^{\mu t}|: \mu\in \sigma(F)\}}.
    \end{equation*} 
    Since $T^t$ is the adjoint of $K^t$, $\sigma(T^t) = \overline{T(K^t)}$, which implies the assertion to be proved. 
\end{proof}

\section{Koopman Observability and Detectability}\label{sec:observability}
The salient difference between the Koopman representation of nonlinear system \eqref{eq:ACE} and a classical infinite-dimensional linear system is the involvement of a state evaluation operator $E: Z\rightarrow \mR^d$. Under \eqref{eq:ACE}, an observer can be formally stated as the following system.
\begin{definition}
    The (ideal) \emph{Koopman--Luenberger observer} for the system in \eqref{eq:ACE} refers to an observer of the form 
    \begin{equation}\label{eq:observer}
        \dot{\hat{z}}(t) = A\hat{z}(t) + L\bra{y(t)-C\hat{z}(t)}, \enspace \hat{x}(t)=E\hat{z}(t). 
    \end{equation}
    Here the bounded linear operator $L\in \mc{B}(\mR^m, Z)$ to be designed is called the \emph{observer gain}. 
    The \emph{RKHS error}, $\eta(t):=\hat{z}(t)-z(t)$, which satisfies $\dot{\eta}(t) = (A-LC)\eta(t)$ (precisely, in the sense of its mild solution: $\eta(t) = \check{T}^t\eta(0)$, where $\{\check{T}^t\}$ is the semigroup generated by $\check A:=A-LC$), determines the \emph{state observation error}, $e(t):=E\hat{z}(t)-x(t)=E\eta(t)$.\footnote{
    Since $\check{A}$ is the generator $A$ perturbed by a bounded operator, we know from semigroup theory \cite{pazy1983semigroups} that $\check{A}$ must generate a strongly continuous semigroup on $Z$ too. When this strongly continuous semigroup to act on any $\eta\in Z$, we formally write $\dot\eta(t) = \check{A}\eta(t)$ as an infinite-dimensional linear system with initial condition $\eta(0)=\eta$, and comprehend it in the sense of its mild solution: $\eta(t)=\check T^t\eta$. 
    }
\end{definition}

\subsection{Koopman Observability in Finite Time}\label{subsec:KO}
We recall from Curtain and Zwart \cite{curtain2020introduction} that the \emph{observability operator} 
\begin{equation*}\Psi^\tau: Z \rightarrow L^2([0, \tau], \mR^m), \enspace z\mapsto CT^{(\cdot)}z \end{equation*}
is the mapping from any $z\in Z$ at initial time to the measured output signals over a duration of $\tau$. 
The infinite-dimensional system $(A, C)$ is \emph{exactly observable in $\tau$} if $\Psi^\tau$ is injective and its inverse is a bounded linear operator on $\ran \Psi^\tau$. As a weaker concept, approximate observability of $(A,C)$ only requires $\ker\Psi^\tau = \{0\}$. The \emph{observability Gramian} is defined as $\Pi^\tau = \Psi^{\tau*}\Psi^\tau = \int_0^\tau T^{t*}C^*CT^t \xD{t}$. 
Now that we are interested in observing $x=Ez$ in $\mR^d$, instead of $z$ in an RKHS, the definitions shall be modified. 
\begin{definition}
    The system $(A, C, E)$ in \eqref{eq:ACE} is said to be \emph{Koopman exactly observable in $\tau$ ($\ms{KXO}_\tau$)} if (i) $\ker \Psi^\tau \subseteq \ker E$, and (ii) $E(\Psi^\tau)^{-1}$, which then becomes a well-defined linear operator from $\ran \Psi^\tau$ to $\mR^d$, is bounded. The system $(A, C, E)$ is \emph{Koopman approximately observable in $\tau$ ($\ms{KAO}_\tau$)} if condition (i) holds. 
\end{definition}

\par The definitions are interpreted as follows: $\ms{KAO}_\tau$ means that any two elements on RKHS that generate identical output signals over $[0,\tau]$ must be read as the same state on the $x$-space; $\ms{KXO}_\tau$ further requires that the difference in the $x$-space must be strictly distinguishable from the difference in the output signals over $[0, \tau]$. 
\begin{remark}[Comparison with ``pointwise observability'' of \cite{mohet2025dual}]\label{rem:pointwise}
	The \emph{pointwise observability} of \cite[Def.~4]{mohet2025dual} requires that distinct states generate distinct outputs, i.e., $\Psi^\tau$ is injective on differences of canonical features. In contrast, $\ms{KAO}_\tau$ requires $Ez=0$ for \emph{every} $z\in Z$ with zero output over $[0,\tau]$. Therefore, $\ms{KAO}_\tau$ implies pointwise observability, but not conversely; $\ms{KAO}_\tau$ is a stronger notion, which is also dependent on the choice of the particular state coordinate representation (which is non-unique since a smooth coordinate transform can alter the $E$ operator). Its merit, however, is that the observer state $\hat z(t)$ generally may not be the canonical feature of a particular $x$, so that a notion posed on directly on $Z$ (instead of on the set of canonical features) is tied to the observer error dynamics and yields the Gramian and Lyapunov conditions below.
\end{remark}

\par In the following theorem, sufficient and necessary conditions for Koopman exact and approximate observability are provided. The proof is effectively identical to Theorem 6.2.6 of \cite{curtain2020introduction} and omitted here. 
\begin{theorem}[Gramian conditions for $\ms{KXO}_\tau$ and $\ms{KAO}_\tau$]
\label{th:Gram.1}
    The system $(A, C, E)$ in \eqref{eq:ACE} is $\ms{KXO}_\tau$ if and only if there exists a constant $\gamma_\tau>0$ such that 
    \begin{equation*}\ip{\Pi^\tau z}{z}_Z \geq \gamma_\tau^2 \ip{Ez}{Ez}_{\mR^d}, \enspace \forall z\in Z. \end{equation*}
    For the system $(A, C, E)$ to be $\ms{KAO}_\tau$, it is necessary and sufficient that under any direct sum decomposition $Z = Z_0 \oplus Z_1$, $Z_0:=\ker E$,  
    $\Pi^\tau = \begin{bmatrix}
        \Pi_{00}^\tau & \Pi_{01}^\tau \\ \Pi_{10}^\tau & \Pi_{11}^\tau
    \end{bmatrix}$ 
    satisfies $\Pi_{11}^\tau \succ 0$ and $\Pi_{00}^\tau - \Pi_{01}^\tau (\Pi_{11}^\tau)^{-1} \Pi_{10}^\tau \succeq 0$. 
\end{theorem}

In view of the goal to design an observer gain $L$ that achieves asymptotic observation, we should extend the observability concepts to infinite time. However $\Psi^\infty$ is not necessarily well defined, i.e., $\Psi^\infty \in \mc{B}(Z, L^2(\mR_+, \mR^m))$ is not guaranteed. 
Hence, we shall make use of the following fact, which states that $\ms{KXO}_\tau$ and $\ms{KXO}_\tau$ does not vary with a ``translation''. 
Specifically, when the semigroup $\{T^t\}_{t\in \mR_+}$ is strongly continuous, there must be an $\omega\in \mR$ and a $c>0$ such that $\|T^t\|\leq c\mr{e}^{\omega t}$ for all $t\in \mR_+$. Then, $\{\mr{e}^{-\lambda t}T^t\}_{t\in \mR_+}$, whose generator is obviously $A_{\leftarrow} = A-\lambda I$, is an exponentially contractive semigroup whenever $\lambda>\omega$. If $\{T^t\}_{t\in \mR_+}$ is already such that $\omega <0$ (i.e., expoentially stable), then it suffices that $A_{\leftarrow}=A$. 
\begin{proposition}\label{th:translation}
    System $(A, C, E)$ is $\ms{KXO}_\tau$ (or $\ms{KAO}_\tau$) if and only if $(A-\lambda I, C, E)$ is $\ms{KXO}_\tau$ (or $\ms{KAO}_\tau$, respectively) for any and hence for all $\lambda\in \mR$. 
\end{proposition}
\begin{proof}
    Let the observability Gramian of system $(A-\lambda I, C, E)$ be $\Pi^\tau_\leftarrow$. Then $\Pi^\tau_\leftarrow = \int_0^\tau e^{-2\lambda t} T^{t*}C^*CT^t \xD{t} \succeq \min\{1, e^{-2\lambda\tau}\}\Pi^\tau$ and similarly $\Pi^\tau_\leftarrow \preceq \max\{1, e^{-2\lambda\tau}\}\Pi^\tau$. By the above theorem, the conclusion holds. 
\end{proof}

\subsection{Koopman Observability in Infinite Time}\label{subsec:KOS}
Now we have a translated observability operator $\Pi_{\leftarrow}^\infty$, which is well-defined and is in $\mc B(Z, L^2(\mR_+, \mR^m))$. 
Koopman approximate observability in infinite time can then be defined, via a virtual translation of the semigroup $\{T^t\}$ that enforces exponential stability. 
\begin{definition}
    Let $\omega_0=\inf\{\omega: \exists c>0 \textrm{ such that } \forall t\in\mR_+, \|T^t\|\leq c\mr{e}^{\omega t}\}$. 
    The system $(A, C, E)$ in \eqref{eq:ACE} is said to be \emph{Koopman exactly observable in infinite time ($\ms{KXO}_\infty$)} if (i) for some (and hence for all) $\lambda>\omega_0$, $\ker\,\Psi^\infty_\leftarrow \subseteq \ker\,E$, and (ii) $E(\Psi_\leftarrow^\infty)^{-1}$, now well-defined from $\ran\,\Psi^\infty_\leftarrow$ to $\mR^d$, is bounded. 
    The system $(A, C, E)$ is \emph{Koopman approximately observable in infinite time ($\ms{KAO}_\infty$)} if condition (i) holds. 
\end{definition}

Condition (i) is equivalent to the statement that the \emph{unobservable subspace} $\{z\in Z: CT^tz = 0,\,\forall t\in\mR_+\} =  \cap_{\tau>0} \ker \Psi^\tau_\leftarrow = \cap_{\tau>0} \ker \Psi^\tau$ is contained in $\ker E$, since $\ker\Psi^\tau$ does not vary with translation from $A$ to $A_\leftarrow$. 
We also remark here that the $\ms{KXO}_\infty$ property is not dependent on the choice of $\lambda$ in the definition, as long as $\lambda$ is sufficiently large. This is because, if under some $\lambda \geq \omega_0$ we have $\Pi^\infty_\leftarrow \succeq \gamma I$, then due to the fact that $\Pi^\infty_\leftarrow$ is the limit in the operator norm topology of $\Pi^\tau_\leftarrow$ (as $\tau\to\infty$), there must be a large $\tau$ such that $\Pi^\tau_\leftarrow \succeq (\gamma-\epsilon)I$ for any small $\epsilon>0$. 
In addition, if the $\ms{KAO}_\infty$ property holds for some $\lambda$ value in the definition, then $\lambda$ can be replaced with any $\lambda'\in (\omega_0, \lambda]$, but not necessarily with $\lambda'>\lambda$. 

In the special case that $\Psi^\infty: z\mapsto CT^{(\cdot)}z$ is indeed a bounded linear operator from $Z$ to $L^2(\mR_+, \mR^m)$, i.e., when the system $(A,C)$ is \emph{output stable}, the observability conditions are easier to verify based on the Gramians. 
To this end, we first define Koopman output stability, which requires the output signal is square-integrable and such an integral is bounded according to the initial state on the $x$-space.  

\begin{definition}
    The system $(A, C, E)$ in \eqref{eq:ACE} is said to be \emph{Koopman output stable ($\ms{KOS}$)} if $\|\Psi^\infty z\|^2\leq \beta^2 |Ez|^2$ holds for all $z\in Z$ for some constant $\beta\geq 0$.\footnote{
    At the first glance, it appears that {\sf KOS} is a stronger requirement than output stability, as $|Ez|^2 \lesssim \|z\|^2$ with $E$ being a bounded linear operator. When using the linear--radial kernel $\rg\kappa$, as revealed in \eqref{eq:isometric}, $z=\rg\kappa(x, \cdot)$ makes $\|z\| \propto |x| = |Ez|$; thus the $\ms{KOS}$ property of $(A,C,E)$ and the output stability of $(A,C)$ are actually equivalent.} 
    In this case, $\Pi^\infty = (\Psi^\infty)^* \Psi^\infty \preceq \beta^2 E^\ast E$.  
\end{definition}

As output stability is connected to an operator Lyapunov equation \cite[Theorem 6.5.2]{curtain2020introduction}, we naturally have the following condition for $\ms{KOS}$. 
\begin{proposition}[Lyapunov equation for $\ms{KOS}$]
\label{th:ALE-KOS}
    $(A, C, E)$ is $\ms{KOS}$ if and only if there exists a nonnegative operator solution $\Pi$, which satisfies $\Pi \preceq \beta^2 E^\ast E$ for some $\beta\geq 0$, to the algebraic Lyapunov equation:
    \begin{equation}
    \ip{\Pi z}{Az} + \ip{Az}{\Pi z} = -\ip{Cz}{Cz}, \enspace \forall z \in \mc{D}(A).
    \end{equation}
    In this case, the minimal bounded nonnegative solution is $\Pi^\infty$, and $(\Pi^\infty)^{1/2} T^t z \rightarrow 0$ as $t\rightarrow \infty$ for all $z\in Z$. 
\end{proposition}

\begin{theorem}[Gramian conditions for $\ms{KXO}_\infty$ and $\ms{KAO}_\infty$]
\label{th:Gram.2}
    Suppose that the system $(A, C, E)$ in \eqref{eq:ACE} is $\ms{KOS}$. 
    The system is $\ms{KXO}_\infty$ if and only if there exists a constant $\gamma_\infty>0$ such that 
    \begin{equation}
    \ip{\Pi^\infty z}{z}_Z \geq \gamma_\infty^2 \ip{Ez}{Ez}_{\mR^d}, \enspace \forall z\in Z. \end{equation}
    For $(A, C, E)$ to be $\ms{KAO}_\infty$, it is necessary and sufficient that under a direct sum decomposition $Z = Z_0\oplus Z_1$, $Z_0:=\ker E$,
    $\Pi^\infty = \begin{bmatrix}
        \Pi_{00}^\infty & \Pi_{01}^\infty \\ \Pi_{10}^\infty & \Pi_{11}^\infty
    \end{bmatrix}$ 
    satisfies $\Pi_{11}^\infty \succ 0$ and $\Pi_{00}^\infty -\Pi_{01}^\infty (\Pi_{11}^\infty)^{-1} \Pi_{10}^\infty \succeq 0$. 
\end{theorem} 

\begin{remark}\label{rem:kerE}
    We characterize the direct sum decomposition used in \Cref{th:Gram.1} and \Cref{th:Gram.2}. Since $Z = \overline{\mr{span}}\{\rg\kappa(x,\cdot): x\in \mbb{X}\}$ and $E\rg\kappa(x,\cdot)=x$, we see that the null space of $E$ is the collection of such linear combinations of kernel functions that the corresponding linear combination of the points is zero:
    \begin{equation*}
    	\ker E = \overline{\mr{span}} \Big\{\sum_{i=1}^n \alpha_i\rg\kappa(x_i, \cdot): 
    	\begin{matrix}
    		\sum_{i=1}^n \alpha_ix_i = 0, \\ \{x_i\}_{i=1}^n \subset \mbb{X}, n\in \mbb{N}
    	\end{matrix} \Big\}.
    \end{equation*}
    We aim to find $Z_1$ that complements $Z_0$. It can easily verified that (i) the component functions $\{e_k\}_{k=1}^d$ are nonzero and linearly independent, and (ii) each $e_k$ is orthogonal to any element in $\ker E$; indeed, $\sum_{i=1}^n \alpha_ix_i = 0$ implies $\ip{e_k}{\sum_{i=1}^n \alpha_i\rg\kappa(x_i, \cdot)} = \sum_{i=1}^n \alpha_i(x_i)_k = 0$. 
    Therefore, by letting 
    \begin{equation}\label{eq:Z1}
        Z_1 = \mr{span}\{e_k\}_{k=1}^d,
    \end{equation} 
    we have $Z_1 \perp Z_0$ and $Z_0\oplus Z_1= Z$.  
    In this way, the condition $\Pi_{11}^\infty \succ 0$ in \Cref{th:Gram.2} is equivalent to the positive definiteness condition: 
    \begin{equation*}
        \Bra{ \alpha_j\alpha_k \ip{e_j}{\Pi_{11}^\infty e_k}}_{j,k=1}^d \succ 0.
    \end{equation*} 
\end{remark}
\begin{remark}\label{rem:homeomorphism-Schur}
    The verification of condition $\Pi_{00}^\infty - \Pi_{01}^\infty (\Pi_{11}^\infty)^{-1} \Pi_{10}^\infty \succeq 0$ is difficult. But it holds automatically if for $Z_0 = \ker E$ and some choice of $Z_1$ that complements $Z_0$, both $Z_0$ and $Z_1$ are invariant under the semigroup $\{T^t\}_{t\in \mR_+}$. 
    A sufficient condition for the invariance of $Z_0$ and $Z_1$ is that the dynamics is linear and diagonal, i.e., $(S^tx)_k = \mr{e}^{\lambda_kt}x_k$. 
    Indeed, in this case, $\sum_{i=1}^n \alpha_ix_i=0$ implies that $\sum_{i=1}^n \alpha_iS^t(x_i)=0$ for all $t\in \mR_+$. 
    Thus, applying the homeomorphism argument in \Cref{prop:stability-spectrum}, when there is a $\rg C^s$-homeomorphism $\psi$ transforming the dynamics $f$ to its diagonalized Jacobian dynamics, we can choose $Z_1 = \mr{span}\{\tilde e_k\}_{k=1}^d$ where $\tilde e_k = \psi^{-1}(e_k)$. 
    Therefore, $\ms{KAO}_\infty$ is equivalent to the positive definiteness condition: 
    \begin{equation*}
        \Bra{ \alpha_j\alpha_k \ip{\tilde e_j}{\Pi_{11}^\infty \tilde e_k}}_{j,k=1}^d \succ 0 .
    \end{equation*}
\end{remark}

\subsection{Koopman Exponential and Strong Detectability} \label{subsec:KD}
For the problem of designing an Koopman observer gain $L$ in \eqref{eq:observer}, it becomes natural that we are only concerned with the $x$ readout from $z$, and hence detectability is defined as the property that a zero output signal implies convergence to $0$ on the state space. 
\begin{definition}
    The system \eqref{eq:ACE} is \emph{Koopman detectable ($\ms{KD}$)} if $CT^tz=0$ ($\forall t\in \mR_+$) implies $\lim_{t\rightarrow \infty} ET^tz=0$. 
\end{definition}
\begin{proposition}[$\ms{KAO}_\infty \Rightarrow \ms{KD}$]
    If $(A, C, E)$ in \eqref{eq:ACE} is $\ms{KAO}_\infty$, and if $\ker E$ is invariant under the semigroup $\{T^t\}_{t\in \mR_+}$ (i.e., for any $z\in Z$, $Ez=0$ implies $ET^tz=0$ for all $t\in \mR_+$), then $(A,C,E)$ is $\ms{KD}$. 
\end{proposition}
\begin{proof}
    Since the $\ms{KAO}_\infty$ property means that $\{z\in Z: CT^tz=0, \forall t\in \mR_+\}$ is contained in $\ker E$, and by the invariance assumption on $\ker E$, $ET^tz=0$ for all $t\in \mR_+$. 
\end{proof}
\begin{remark}
    As argued in \Cref{rem:homeomorphism-Schur}, $\ker E$ is invariant under the semigroup $\{T^t\}$ if the system is a $d$-dimensional diagonal linear one, or brought so by a $\rg C^s$-homeomorphism $\psi$. 
    Therefore, if a $\rg C^s$-homeomorphism $\psi$ exists such that $\psi\circ S^t\circ \psi^{-1}: x\mapsto \mr{e}^{tF}x$ and $F=\mr{D}f(0)$ is diagonalizable, then $\ms{KAO}_\infty \Rightarrow \ms{KD}$.
\end{remark}

\begin{definition}
    The system \eqref{eq:ACE} is said to be \emph{Koopman exponentially detectable ($\ms{KED}$)} if there exists an $L\in \mc{B}(\mR^m, Z)$ such that $\check{A}=A-LC$ is the infinitesimal generator of a strongly continuous semigroup $\{\check{T}^t\}$ satisfying $\|E\check{T}^t\|\leq c\mr{e}^{\alpha t}$ for all $t\in \mR_+$ with constants $c>0$ and $\alpha<0$. 
    The system \eqref{eq:ACE} is \emph{Koopman strongly detectable ($\ms{KSD}$)} if there exists an $L\in \mc{B}(\mR^m, Z)$ such that for all $z\in Z$, $\lim_{t\rightarrow\infty} E\check{T}^t z = 0$.
\end{definition}

The requirement of Koopman exponential or strong detectability is weaker than exponential or strong detectability of the infinite-dimensional system $(A,C)$ in the original non-Koopman sense, as discussed in the following remark.

\begin{remark}[Koopman detectability versus spectral-gap design]\label{rem:spectral-gap}
	By \cite[Theorem 8.1.3]{curtain2020introduction}, for $(A,C)$ with a finite-rank $C$ to be $\beta$-exponentially detectable by a bounded $L$, the part of $\sigma(A)$ to the right of $\Re\lambda=\beta-\epsilon$ (for some $\epsilon>0$) must consist of at most finitely many eigenvalues with finite-dimensional eigenspaces. This spectral-gap condition underlies the design of \cite{mohet2025dual} that seeks exponential decay of RKHS error at a prescribed rate $\beta$. This can fail, however, when $A$ has infinitely many eigenvalues, or a continuous spectrum, in $\{\lambda: \Re\lambda\geq\beta\}$, as is the case for limit cycle or chaotic systems. 
	$\ms{KED}$ and $\ms{KSD}$ only require the readout $E\check T^t\eta$ to decay, not the error components invisible through $E$. The expense is that no spectral gap is certified and hence we do not seek to guarantee a prescribed rate on the RKHS.  
\end{remark}

\begin{proposition}[$\ms{KXO}_\tau \wedge \ms{KOS} \Rightarrow \ms{KSD}$]
    Suppose that $(A, C, E)$ is $\ms{KXO}_\tau$ for some $\tau>0$ and $\ms{KOS}$. Then $(A,C,E)$ is $\ms{KSD}$. 
\end{proposition}
\begin{proof}
    Decompose $Z$ as the direct sum $Z_0\oplus Z_1$, where $Z_0=\ker E$, and correspondingly write $\eta=\eta_0+\eta_1$ for any $\eta\in Z$ and $C=[C_0, C_1]$. 
    Under the given hypotheses, by \Cref{th:Gram.2}, there exists a $\gamma_\infty>0$ such that $\Pi^\infty_{11} \succeq \gamma_\infty I$. Let $L = \begin{bmatrix} 0 & 0 \\ 0 & (\Pi^\infty_{11})^{-1} \end{bmatrix} C^\ast$, which is now a bounded operator.
    We define a Lyapunov functional $v(\eta) = \ip{\eta}{\Pi^\infty \eta}$ and verify 
    \begin{equation*}
        \dot{v} = \ip{\eta}{(\Pi A + A^*\Pi)\eta} - 2|C_1\eta_1|^2 \leq -|C\eta|^2,
    \end{equation*}
    due to the $\ms{KOS}$ condition and \Cref{th:ALE-KOS} which justifies $\Pi A + A^*\Pi = -C^*C$. The integration of the above inequality gives 
    \begin{equation*}
    	\int_0^\infty |C\eta(t)|^2 \xD{t}\leq \ip{\eta(0)}{\Pi^\infty \eta(0)}.
    \end{equation*}
    As an implication, $\int_t^\infty |C\eta(t')|\xD{t'} \to 0$ as $t\to\infty$. 
    Now, because $(A, C)$ is $\ms{KXO}_\tau$, we must have that $(A-LC, C)$ is $\ms{KXO}_\tau$ too, and therefore, $|E\eta(0)|^2 \lesssim \int_0^\infty |C\eta(t)|^2 \xD{t}$. The redefinition of the starting time gives 
    \begin{equation*}
    	|E\eta(t)|^2 \lesssim \int_0^\infty |C\eta(t+t')|^2 \xD{t'} = \int_t^\infty |C\eta(t')|^2 \xD{t'} \to 0
    \end{equation*}
    as $t\to \infty$. This is the $\ms{KSD}$ property to be proved. 
\end{proof}

\begin{proposition}[$\ms{KXO}_\tau \wedge \ms{KOS} \wedge Z_1 \text{ invariance}\Rightarrow \ms{KED}$]
    Suppose that $(A, C, E)$ is $\ms{KXO}_\tau$ for some $\tau>0$ and $\ms{KOS}$, and in addition, $Z_1$ is forward-invariant under $\{\check T^t\}_{t\in \mR_+}$. Then $(A,C,E)$ is $\ms{KED}$. 
\end{proposition}
\begin{proof}
    Previously we have proved that $|E\eta(t)|^2\to 0$ as $t\to \infty$ for all $\eta\in Z$. Since $Z_1$ is a $d$-dimensional subspace of $Z$ that complements $Z_0=\ker E$, $|E\eta(t)|$ can be replaced by $|\eta_1(t)|$. 
    Due to the invariance of $Z_1$, $\eta_1(t)$ evolves under a projected semigroup $\{\check T_1^t\}_{t\in \mR_+}$. Hence, $|\check T_1^t\eta|\to 0$ for any $\eta\in Z$. By the Datko--Pazy theorem, $\{\check T_1^t\}$ must be an exponentially stable semigroup. Hence, $\|E\check T^t\| \lesssim \|\check T_1^t\| \leq c\mr{e}^{\alpha t}$ for some $c>0$ and $\alpha<0$. 
\end{proof}

To the end of computing an observer gain $L$ that drives the evolution of observation error $\eta(t)$ so that $E\eta(t)\to 0$ (and hence $\hat x(t)-x(t)\to 0$), an operator Lyapunov inequality condition is established in the following theorem. 
\begin{theorem}[Lyapunov inequality for the observer]
\label{th:ALE-Observer}
    If the system $(A, C, E)$ is $\ms{KED}$, then for some $L\in \mc{B}(\mR^m, Z)$, there exists $0 \preceq P\in \mc{B}(Z)$ that is positive on $Z_1 = Z_0^\perp$, such that 
    \begin{equation}\label{eq:oALE.observer}
        \begin{aligned}&\ip{\check{A}\eta}{P\eta} + \ip{P\eta}{\check{A}\eta} \leq -k\ip{E\eta}{E\eta}, \\ &\forall \eta\in \mc{D}(A)=\mc{D}(\check{A})\end{aligned}
    \end{equation} 
    holds for some $k>0$. Conversely, if there exists such a solution to \eqref{eq:oALE.observer} with some $L\in \mc{B}(\mR^m, Z)$, then the system $(A, C, E)$ in \eqref{eq:ACE} is at least $\ms{KSD}$; if, moreover, $Z_0=\ker E$ and $Z_1$ are both invariant, then the system is $\ms{KED}$; if, furthermore, $\ran C^* \subseteq Z_1$, the solution $L\in \mc{B}(\mR^m, Z)$ can be chosen such that $\ran L \subseteq Z_1$.\footnote{The interpretation of $\ran L \subseteq Z_1$ is natural: in the observer \eqref{eq:observer}, the innovation term $y - C\hat{z}$ does not need to cause the velocity $\dot{\hat{z}}(t)$ to have a component in $Z_0=\ker E$ which has no effect on the state estimate. }
\end{theorem}
\begin{proof}
    If $\ms{KED}$ holds, then $P = \int_0^\infty \check{T}^{t\ast} E^\ast E \check{T}^t \xD{t}$ is indeed bounded, nonnegative, and positive on $Z_1$. With the fact that $\xD{\check{T}^t \eta}/\xD{t}=A\check{T}^t\eta =\check{T}^tA \eta$ for all $\eta\in \mc{D}(A)$, the Lyapunov inequality \eqref{eq:oALE.observer} is verified with $k=1$. 
    Conversely, if the Lyapunov equation is solved by an operator $P$ as required, then the examination of Lyapunov functional $v(\eta) = \ip{\check{T}^t \eta}{P\check{T}^t \eta}$ verifies that 
    \begin{equation*}
        \int_0^\infty |E \check{T}^t \eta|^2\xD{t} \leq \ip{\eta(0)}{P\eta(0)}\text{, i.e., }\int_0^\infty |E\eta(t)|^2\xD{t} < \infty 
    \end{equation*}for any given $\eta\in Z$. This implies $E\eta(t)\to 0$, i.e., the system is $\ms{KSD}$. 
    The $\ms{KED}$ property in the case with invariant $Z_0$ and $Z_1$ is proved in the same way as in the previous theorem. In this case, $A$ is block-diagonal, $E^*E$ only contains the lower-right block (as $Z_0=\ker E$), and $\ran C^*\subseteq Z_1$ implies $C = \begin{bmatrix} 0, * \end{bmatrix}$, then $A-LC$ is upper block triangular. 
    The satisfiability of \eqref{eq:oALE.observer} is not affected by the off-diagonal block, i.e., the upper block of $L = \begin{bmatrix} 0 \\ * \end{bmatrix}$ can be set to zero. That is, $\ran L \subseteq \ran C^*$. 
\end{proof}

\begin{remark}
    The condition that $\ran C^* \subseteq Z_1$ is interpreted as follows. By definition, the rows of $C$ correspond to $h_1, \cdots, h_m\in Z$, and hence the condition that left block of $C$ is zero means that $\ip{h_j}{z}=0$ for all $z\in \ker E$ and $j=1,\cdots,m$. 
    Then, by \Cref{rem:kerE} where $\ker E$ is expressed, we know that the foregoing condition is equivalent to 
    \begin{equation*}
    \sum_{i=1}^n \alpha_i h_j(x_i) = 0 \text{ whenever } \sum_{i=1}^n \alpha_i x_i=0
    \end{equation*}
    for any choice of $x_1, \cdots, x_n\in \mbb{X}$, $\alpha_1, \cdots,\alpha_n\in \mR$, and $n\in \mN$. This condition is satisfied when $h_1, \cdots, h_m$ are linear functions of $x$, namely, when the measured outputs are some state components or their linear combinations. 
\end{remark}

\subsection{Criteria for Riesz Spectrum Operators}\label{subsec:KO.test}
For a special class of generators $A$, namely when $A$ is a \emph{Riesz spectrum operator}, the $\ms{KAO}_\infty$ and $\ms{KED}$ properties can be related to the spectral attributes of $A$. The following results are presented here for theoretical completeness and comparison to the classical infinite-dimensional systems theory \cite{curtain2020introduction}, although it must be pointed out that for a given nonlinear system, it is generally difficult to verify whether its infinitesimal generator is a Riesz spectrum operator. 
\begin{definition}
    A closed and densely defined operator $A$ on $Z$ is called a \emph{Riesz spectrum operator} if it has simple eigenvalues $\{\lambda_j\}_{j=1}^\infty$ and associated eigenvectors $\{\phi_j\}_{j=1}^\infty$ which forms a Riesz basis\footnote{That is, $\mr{span}\{\phi_j\}_{j=1}^\infty$ is dense in $Z$ and there exists constants $0<c_1\leq c_2$ such that for any $n\in \mN$ and $\alpha_1, \dots, \alpha_n\in \mR$, $c_1\sum_{j=1}^n |\alpha_j|^2 \leq \|\sum_{j=1}^n \alpha_j\phi_j \|^2 \leq c_2\sum_{j=1}^n |\alpha_j|^2$. }
    and $\{\lambda_j\}_{j=1}^\infty$ has at most finitely many limit points. As such, we have a biorthogonal basis $\{\varphi_j\}_{j=1}^\infty$ such that $\ip{\phi_j}{\varphi_{j'}}_Z = \delta_{jj'}$, which determines the representation of the operator as:
    \begin{equation}\label{eq:Riesz.spectral}
        A = \sum_{j=1}^\infty \lambda_j\ip{\cdot}{\varphi_j}_Z \phi_j, \enspace z\in \mc{D}(A),
    \end{equation}
    where $\mc{D}(A) =\{z\in Z:\sum_{j=1}^\infty |\lambda_j|^2|\ip{z}{\varphi_j}|^2 <\infty\}$. 
\end{definition}

 \begin{theorem}
 \textit{(Test of $\ms{KAO}_\infty$ for Riesz spectral systems)}
    Suppose that $A$ is a Riesz spectrum operator with representation \eqref{eq:Riesz.spectral}, $C\in \mc{L}(Z,\mR^m)$, and $E\in \mc{L}(Z, \mR^d)$. Then the system $(A, C, E)$ is Koopman approximately observable in infinite time if and only if all eigenfunctions $\phi_j$ that satisfy $C\phi_j=0$ also satisfy $E\phi_j=0$, i.e., if and only if any eigenfunction $\phi_j$ not belonging to $\ker E$ must make $C\phi_j\neq 0$. 
\end{theorem}
\begin{proof}
    Based on the definition of $\{\phi_j,\varphi_j\}$, we have for all $z\in Z$ that 
    \begin{equation*}
        CT^tz = \sum_{j=1}^\infty e^{\lambda_j t}\ip{z}{\varphi_j} C\phi_j     
    \end{equation*}
    and hence the unobservable subspace of $Z$ is expressed as 
    \begin{equation*}
        \bigcap_{\tau>0} \ker CT^t = \BRA{z\in Z:  \ip{z}{\varphi_j} C\phi_j =0, \,\forall j\in\mN\backslash\{0\} }, 
    \end{equation*}
    which is further equal to $\mr{span}\{\phi_j:C\phi_j=0\}$. This set is equal to $\ker E$ if and only if all eigenfunctions $\phi_j$ of the Riesz spectral operator $A$ such that $C\phi_j=0$ is in $\ker E$. 
\end{proof}

The theorem below gives the criterion for $\ms{KED}$ when $A$ is a Riesz spectral operator. 
\begin{theorem}
\textit{(Test of $\ms{KED}$ for Riesz spectral systems)}
    Suppose that $(A,C,E)$ is a Riesz spectral system with representation of $A$ given in \eqref{eq:Riesz.spectral}. For $(A,C,E)$ to be $\ms{KED}$ at rate $\alpha<0$, it is necessary and sufficient that there exists an $\epsilon>0$, such that $\sigma_{\alpha-\epsilon}^+(A) = \{\lambda_j: \Re \lambda_j\geq \alpha-\epsilon \,\mr{and}\, E\phi_j\neq 0\}$ contains finitely many eigenvalues and $C\phi_j\neq 0$ for all $\lambda_j\in \sigma_{\alpha-\epsilon}^+(A)$. 
\end{theorem}

\section{Koopman--Luenberger Observer Synthesis}\label{sec:synthesis}
Now we address the computational problem of synthesizing the observer, assuming that the nonlinear system \eqref{eq:sys} can be represented as \eqref{eq:ACE} through a RKHS formulation and that the resulting system triplet $(A,C,E)$ possesses the $\ms{KED}$ property. Our need is to find a nonnegative operator $P$ as well as the observer gain $L$ such that the Lyapunov inequality holds. 
We first rewrite the inequality \eqref{eq:oALE.observer} as
\begin{equation}\label{eq:oALE.main}
    \ip{z}{(PA + A^\ast P - YC - C^\ast Y^\ast + kE^\ast E)z} \leq 0, \ z\in \mc{D}(A), 
\end{equation}
where $Y:=PL$. The introduction of such a product variable is a standard trick in formulating the controller/observer synthesis problems for linear systems via linear matrix inequalities; see Caverly and Forbes \cite{caverly2019lmi}. 
The parameter $k>0$ in \eqref{eq:oALE.observer} is user-specified to indicate a desired rate of observation error decay, with which the nonnegative operator $P$ provides a certificate for the decrease of state observation error. Hence, larger $k$ corresponds to more aggressive observer. 
To ensure that $P$ is positive definite on $Z_1$ (which is $d$-dimensional), we impose a partial positive definiteness constraint: 
\begin{equation}\label{eq:P.psd}
    P|_{Z_1} \succeq I. 
\end{equation} 

\subsection{Compactified and Empirical Lyapunov Inequality}\label{subsec:compactified}
\par To obtain $P$ and $Y=PL$ from \eqref{eq:oALE.main}, we need to rely on some data sample to approximate the generator $A$ as a bounded operator and numerically solve the inequality as a matrix inequality. 
Since a bounded, finite-rank operator is in nature different from the infinitesimal generator $A$ as an unbounded operator on a Hilbert space, the approximation error must be characterized carefully. 
Here, we adopt the idea of resolvent-type EDMD (RT-EDMD), proposed in Meng et al. \cite{meng2026resolvent}, to estimate the semigroup generator from data. 
\begin{definition}
    The \emph{resolvent} of the infinitesimal generator $A$ of the semigroup $\{T^t\}_{t\in \mR_+}$ is $R_\lambda=(\lambda I-A)^{-1}\in \mc{B}(Z)$ (defined for $\lambda\in \mC$ and $\lambda \notin \sigma(A)$). The \emph{Yosida approximation} of $A$ is $A_\lambda = \lambda^2 R_\lambda - \lambda I \in \mc{B}(Z)$, defined for sufficiently large $\lambda>0$. 
    Since $R_\lambda$ for large $\lambda>0$ is given by the Laplace transform: $R_\lambda = \int_0^\infty \mr{e}^{-\lambda t}T^t\xD{t}$, we refer to 
    \begin{equation}\label{eq:Yosida}
        A_{\lambda, \tau} = \lambda^2 \int_0^\tau \mr{e}^{-\lambda t}T^t\xD{t} - \lambda I
    \end{equation} 
    as the \emph{finite-time Yosida approximation}. 
\end{definition}

In \cite{meng2026resolvent}, the following error bounds for finite-time Yosida approximation was established for the Koopman semigroup defined on the Banach space $C(\mbb{X})$. 
The same conclusion is generalizable to the RKHS formulation through essentially the same proofs, cf. Th. 3.1 and Proposition 4.2 of \cite{meng2026resolvent}. 
\begin{lemma}\label{lem:Yosida}
    Suppose that $\|T^t\|\lesssim \mr{e}^{\omega t}$. For all $z\in \mc{D}(A^2)$, we have 
    \begin{equation*}
        \begin{aligned}\|(A_{\lambda, \tau}-A)z\| &\lesssim \lambda^{-1}(\|z\|+\|Az\|) + \lambda \mr{e}^{-\lambda \tau}\|z\| \\ &\lesssim (\lambda^{-1}+\tau^{-1})(\|z\| + \|Az\|)\end{aligned}
    \end{equation*}
    and hence 
    \begin{equation}
    \begin{aligned}\|T^t\eta - \exp(tA_{\lambda,\tau})\eta\| \lesssim (\lambda^{-1}+\tau^{-1})(\|\eta\| + \|A\eta\|)t \mr{e}^{\omega t}.\end{aligned}
    \end{equation} 
\end{lemma}

Thus, the operator Lyapunov inequality \eqref{eq:oALE.main} is approximated as 
\begin{equation*}
    \begin{aligned}\ip{z}{(PA_{\lambda,\tau} + A_{\lambda,\tau}^\ast P - YC - C^\ast Y^\ast + kE^\ast E)z} \leq 0, & \\ \forall z\in \mc{D}(A). \end{aligned}
\end{equation*}
Now all operators are bounded. Hence, the Yosida-approximated Lyapunov inequality is:
\begin{equation}\label{eq:oALE.Yosida}
    PA_{\lambda,\tau} + A_{\lambda,\tau}^\ast P - YC - C^\ast Y^\ast + kE^\ast E \preceq 0. 
\end{equation}
In view of \eqref{eq:Yosida}, a data-based approximation of the bounded operator $A_{\lambda, \tau}$ is needed. This approximation problems can be formulated as a regression on a finite number ($n$) of trajectories, each of a duration of $\tau$, starting from independently collected initial points $\{x_i\}_{i=1}^n$ from a probability measure $\nu$ supported on $\mbb{X}$, i.e., the available data is $\{x_i(t) = S^t(x_i): t\in [0, \tau], i\in \{1, \cdots, n\}\}$, $\{x_i\}_{i=1}^n \overset{\text{i.i.d.}}{\sim} \nu$. 
For any $x\in \mbb X$, denote $\rg\kappa_x = \rg\kappa(x,\cdot)$ and
\begin{equation*}\begin{aligned}A_{\lambda, \tau}\rg\kappa_x =: \rg\kappa'_x = \lambda^2 \int_0^\tau \mr{e}^{-\lambda t} \rg\kappa_{x(t)} \xD{t} - \lambda \rg\kappa_{x(0)} \in Z.\end{aligned}\end{equation*}
Since $A_{\lambda,\tau}$ is not compact while any data-based approximation $\hat{A}_{\lambda,\tau}$ must be finite-rank, $\hat{A}_{\lambda,\tau}$ cannot converge to $A_{\lambda, \tau}$ even in a strong (pointwise) sense.\footnote{
    An idea in \cite{kostic2022learning, kostic2023sharp} is to consider the operator as from $Z$ to $L^2(\mbb{X})$ instead, namely to compose the operator with the compact inclusion map from $Z$ to $L^2(\mbb{X})$, which is then a compact operator. The estimate must use a reduced rank formulation and its error analysis involves the spectrum of the $R_{j,\tau}$ below. 
    However, in general, solving control-theoretic problems in a Koopman-based manner does not necessarily need an explicit estimate of the Koopman operator/generator. 
}

Now we note that 
\begin{equation*}
	A_{\lambda,\tau} = \Bra{\int_{\mbb{X}} \rg\kappa'_x\times \rg\kappa_x \xD{\nu(x)}} \Bra{ \int_{\mbb{X}} \rg\kappa_x\times \rg\kappa_x \xD{\nu(x)}} ^{-1} , 
\end{equation*}
where by the outer product form $z_1\times z_2$ we refer to a rank-$1$ operator on $Z$, specified by $(z_1\times z_2)z = \ip{z_2}{z}z_1$. 
\begin{definition}
    We refer to 
    \begin{equation*}
    	\begin{aligned}
    		R_{\lambda,\tau} &= \frac{1}{\nu(\mbb{X})} \int_{\mbb{X}} \rg\kappa'_x\times \rg\kappa_x \xD{\nu(x)} \text{ and } \\ S &= \frac{1}{\nu(\mbb{X})} \int_{\mbb{X}} \rg\kappa_x\times \rg\kappa_x \xD{\nu(x)},
    	\end{aligned}
    \end{equation*}
    namely the operators that map any given $g\in Z=\mc{H}_{\rg\kappa}(\mbb{X})$ to 
    \begin{equation*}
        \begin{aligned}R_{\lambda,\tau}g &= \frac{1}{\nu(\mbb{X})}\int_{\mbb{X}} g(x) \rg\kappa'_x \xD{\nu(x)} \text{ and } \\ Sg &= \frac{1}{\nu(\mbb{X})}\int_{\mbb{X}} g(x)\rg\kappa_x \xD{\nu(x)}, \end{aligned}
    \end{equation*}
    respectively, as the \emph{$\rg\kappa'$-integral operator} and \emph{$\rg\kappa$-integral operator}, respectively. 
    Their sample average approximations, namely the empirical operators, are respectively denoted as 
    \begin{equation*}
        \hat{R}_{\lambda,\tau} = \frac{1}{n}\sum_{i=1}^n \rg\kappa'_{x_i} \times \rg\kappa_{x_i} \text{ and } \hat{S} = \frac{1}{n}\sum_{i=1}^n \rg\kappa_{x_i} \times \rg\kappa_{x_i}.
    \end{equation*} 
\end{definition}

Now $R_{\lambda,\tau}$ and $S$ are both compact operators, and $\hat{R}_{\lambda,\tau}$ and $\hat{S}$ are ``good'' finite-rank approximations that converges to $R_{\lambda,\tau}$ and $S$ in operator norm (i.e., uniformly) at the large data limit (see next subsection). 
Although $A_{\lambda, \tau} = R_{\lambda,\tau}S^{-1}$ is non-compact and cannot be hoped to be well approximated by replacing with an empirical expression, we can transform the Yosida-approximated Lyapunov inequality \eqref{eq:oALE.Yosida} into the following equivalent form, by left-multiplying $S$ and right-multiplying $S$ (and noting that $S$ is self-adjoint):
\begin{equation}\label{eq:oALE.compact}
    \begin{aligned}
    	SPR_{\lambda,\tau} + R_{\lambda,\tau}^\ast PS - SYCS - SC^*Y^*S + kSE^*ES \preceq 0.
    \end{aligned}
\end{equation}
The above inequality contains all its left-hand side terms as compact operators, and hence is said to be \emph{compactified}. 

With empirically approximated integral operators $\hat{R}_{\lambda,\tau}$ and $\hat{S}$ of a finite rank, we are interested in solving the following \emph{empirical} version of the operator Lyapunov inequality:
\begin{equation}\label{eq:oALE.empirical}
    \begin{aligned}
    	\hat S \hat{P} \hat R_{\lambda,\tau} + \hat R_{\lambda,\tau}^\ast \hat P\hat S - \hat S \hat YC \hat S - \hat S C^*\hat Y^* \hat S + k\hat S E^*E \hat S \preceq 0. \end{aligned}
\end{equation}
Its solution $(\hat{P}, \hat{Y})$ therefore must be restrictable onto some finite-dimensional subspace. 
That is, if $(\hat{P}, \hat{Y})$ is a solution pair to \eqref{eq:oALE.empirical}, then the projection of $\hat{P}$ onto $\mr{span}\{\rg\kappa_{x_i}\times \rg\kappa_{x_j}\}_{i,j=1}^n$ and the projection of $\hat{Y}$ onto $\mr{span}\{e_i \times \ms{e}_j \}_{i,j=1}^n$, must form a solution too. 
It then suffices to solve \eqref{eq:oALE.empirical} in their matrix representations, which is to be covered in \Cref{subsec:elementary}. 
With this solution, by letting $\hat{L} = \hat{P}^{-1}\hat{Y}$, an observer in the Yosida-approximate form is available:
\begin{equation}\label{eq:observer.Yosida}
    \dot{\hat z}_{\lambda,\tau} = A_{\lambda,\tau} \hat{z}_{\lambda,\tau} + \hat{L}(y-C\hat{z}_{\lambda,\tau}), \enspace \hat x_{\lambda,\tau} = E\hat{z}_{\lambda,\tau}.
\end{equation}
Indeed, although the ideal Koopman--Luenberger observer in the form of \eqref{eq:observer} should use the exact $A$ as in the model \eqref{eq:ACE}, because $A$ is not computable, it can be only approximated as a finite-time Yosida approximation $A_{\lambda, \tau}$.

\subsection{Learning Error and Its Influence on State Observation}\label{subsec:learning.error}
The replacement of $S$ by $\hat S$ introduces a learning error. Since $\hat S$ and $S$ are compact, the learning error is characterized as a probabilistic bound in operator norm. The following conclusion is known from Kostic et al. \cite{kostic2022learning}. 
\begin{lemma}
    We have, with probability at least $1-\delta$ over the random sampling of $\{x_i\}_{i=1}^n \overset{\text{i.i.d.}}{\sim} \nu$, that
    \begin{equation*} \frac{\|\hat{R}_{\lambda,\tau} - R_{\lambda,\tau}\|}{\|A_{\lambda,\tau}\|}, \, \|\hat{S}-S\| \lesssim \frac{1}{n}\log\frac{8n^2}{\delta} + \sqrt{\frac{1}{n}\log\frac{8n^2}{\delta}}.\end{equation*}
\end{lemma}
Hence, given any $\ve>0$, there exists a corresponding $n_{\ve, \delta}$ such that when the sample size exceeds $n_{\ve,\delta}$, we have a confidence of $1-\delta$ that both operator errors in norm do not exceed $\epsilon$. As such, if the true system allows a solution to \eqref{eq:oALE.compact}, which we denote as $(P, Y)$, then this solution must solve the empirical version \eqref{eq:oALE.empirical} with the right-hand side relaxed by a multiple of $\ve (S+|R_{\lambda, \tau}|)$, by accounting for the operator norm errors in $\hat{R}_{\lambda,\tau}$ and $\hat{S}$.\footnote{
Here by $|R|$ we refer to the matrix obtained via the polar decomposition of $R$ as a bounded operator in the Hilbert space $Z$: $R = U|R|$, where $U$ is a partial isometry. One can take, $|R| = (R^*R)^{1/2}$, for example.}
Such a relaxed empirical equation is hence feasible. 
Conversely, any solution of the empirical equation \eqref{eq:oALE.empirical}, denoted as $(\hat{P}, \hat{Y})$, when plugged into the original formulation \eqref{eq:oALE.compact}, must also result in an error bounded by such a term of $\ve (\hat{S}+|\hat{R}_{\lambda, \tau}|)$. We thus reached at the following conclusion. 
\begin{corollary}\label{cor:oALE.compact.error}
    Suppose that \eqref{eq:oALE.compact} admits a solution $(P,Y)$. Then for any given $\ve>0$ sufficiently small and $\delta\in(0,1)$, there exists $n_{\ve, \delta}\in \mN$, such that when $n>n_{\ve, \delta}$, with probability at least $1-\delta$ over random sampling, if \eqref{eq:oALE.empirical} is feasible, then any solution $(\hat{P}, \hat{Y})$ satisfies:
    \begin{equation*}
        \begin{aligned}&S \hat P R_{\lambda,\tau} + R_{\lambda,\tau}^* \hat P S - S\hat Y C S - SC^* \hat Y^* S \\ &\quad + kSE^*ES  \preceq \ve\bra{1+\|A_{\lambda,\tau}\|}S. \end{aligned}
    \end{equation*} 
\end{corollary}

To see the effect of the this learning error on the observation performance, we should analyze the error between the observer states of the Yosida-approximated observer, $\hat{z}_{\lambda,\tau}$, and the state $z_{\lambda,\tau}$ that evolves under the Yosida-approximated generator:
\begin{equation}\label{eq:ACE.Yosida}
    \dot{z}_{\lambda,\tau} = A_{\lambda,\tau} z_{\lambda,\tau}. 
\end{equation}
Clearly, their difference $\eta_{\lambda,\tau}(t) := \hat{z}_{\lambda,\tau}(t) - z_{\lambda,\tau}(t)$ evolves under the ``closed-loop'' operator $A_{\lambda,\tau}-\hat LC$. In order to establish a bound on the state observation error based on the evolution of $\eta_{\lambda,\tau}$, we introduce the following ``dissipativity by sampling'' condition. 
\begin{definition}\label{def:dissipative}
    An infinite-dimensional system $\dot{z}(t)=Az(t)$ on the RKHS with kernel $\rg\kappa$ is said to be \emph{dissipative by sampling}, if $SA + A^*S\preceq 0$, where $S = \frac{1}{\nu(\mbb{X})}\int_{\mbb{X}} \rg\kappa_x\times \rg\kappa_x \xD{\nu(x)}$ is the sampling operator. An infinite-dimensional system with input: $\dot{z}(t) = Az(t)+Bu(t)$, or simply called $(A,B)$, is said to be \emph{closed-loop dissipative by sampling under controller gain $K$}, if the system with generator $A-BK$ is dissipative by sampling. An infinite-dimensional system with output: $\dot{z}(t) = Az(t)$, $y(t) = Cz(t)$, or simply called $(A,C)$, is \emph{closed-loop dissipative by sampling under observer gain $L$}, if its dual system $(A^*, C^*)$ (as an infinite-dimensional system with input) is closed-loop dissipative by sampling under controller gain $L^*$. 
\end{definition}
\begin{remark}\label{rem:dissipativity}
    The semidefiniteness inequality on the dissipativity of sampling operator $S$ is a very natural condition that rules out possible failure of data-based estimation (learning) techniques on stability, control, and observation. For the system governed by the generator $A$, it is clear that the dissipative inequality $SA + A^*S\preceq 0$ guarantees that $z\mapsto \ip{z}{Sz}$ is a monotonically non-increasing function of time. 
    If the dissipativity holds on the dual system, i.e., if $SA^\ast + AS \preceq 0$, then for the system $\dot{z}(t)=A^*z(t)$, if the initial condition $z(0)\in \ran S$, then by writing $z(t)=S\zeta(t)$ and considering the evolution of $\zeta(t)$, we have $\xDD{}{t} \ip{\zeta(t)}{S\zeta(t)} \leq 0$ and hence $\ip{z(t)}{S^{-1}z(t)}$ remains bounded. 
\end{remark}

\begin{proposition}\label{th:observer.performance}
    Suppose that the system $(A_{\lambda,\tau}, C)$ is closed-loop dissipative by sampling under observer gain $\hat L$. 
    For any sufficiently small $\ve>0$, with $1-\delta$ probability over the independent sampling of more than $n_{\ve, \delta}$ data trajectories by probability measure $\nu$, the observer synthesized from $(\hat{P}, \hat{Y})$ guarantees bounded time-averaged $L^2$-error in the sense of:
    \begin{equation*} 
        \frac{1}{\theta} \int_0^\theta |E(\hat{z}_{\lambda,\tau}(t) - z_{\lambda,\tau}(t))|^2 \xD{t} \lesssim \frac{\ve}{k}\bra{1+\|A_{\lambda,\tau}\|} + \frac{1}{\theta}. 
    \end{equation*}
    for any $\theta>0$. 
\end{proposition}
\begin{proof}
    As we write $\eta_{\lambda,\tau}(t) = \hat{z}_{\lambda,\tau}(t) - z_{\lambda,\tau}(t)$, clearly, we have $\dot{\eta}_{\lambda,\tau}(t) = (A_{\lambda, \tau}-\hat{L}C)\eta_{\lambda,\tau}(t) = \check{A}_{\lambda, \tau} \eta_{\lambda, \tau}(t)$. 
    By defining a Lyapunov functional on $\eta_{\lambda,\tau}$:
    \begin{equation*}
        v(t) = \ip{\eta_{\lambda,\tau}(t)}{\hat P\eta_{\lambda,\tau}(t)} = \ip{\check T_{\lambda, \tau}^t \eta_{\lambda,\tau}}{\hat P \check T_{\lambda, \tau}^t \eta_{\lambda,\tau}},
    \end{equation*}
    we have
    $\dot{v}(t)=\ip{\eta_{\lambda,\tau}(t)}{(\hat P\check A_{\lambda,\tau} + \check A_{\lambda,\tau}^*\hat P)\eta_{\lambda,\tau}(t)}.$
    As \Cref{cor:oALE.compact.error} reveals, we have
    \begin{equation*}
        \begin{aligned}\dot{v}(t) \leq &-k\ip{\eta_{\lambda,\tau}(t)}{E^\ast E \eta_{\lambda,\tau}(t)} \\ &+ \ve \bra{1+\|A_{\lambda,\tau}\|} \ip{\eta_{\lambda,\tau}(t)}{S^{-1}\eta_{\lambda,\tau}(t)}.\end{aligned}
    \end{equation*}
    By the hypothesis of closed-loop dissipativity by sampling under observer gain $\hat L$, the discussions in \Cref{rem:dissipativity} shows that $\ip{\eta_{\lambda,\tau}(t)}{S^{-1}\eta_{\lambda,\tau}(t)}$ will remain bounded. Therefore, integrating the inequality above on $t\in [0, \theta]$ for any $\theta>0$ and taking the time average, we have 
    \begin{equation} 
        \frac{k}{\theta}\int_0^\theta |E\eta_{\lambda,\tau}(t)|^2 \xD{t} - v(0) + v(\theta) \leq \ve \bra{1+\|A_{\lambda,\tau}\|} \cdot \mr{const}. 
    \end{equation}
    The conclusion to be proved thus follows. 
\end{proof}

\subsection{Koopman--Luenberger Observer with a Data-Based Finite-Order Approximation}\label{subsec:performance}
The Yosida-approximate observer \eqref{eq:observer.Yosida} remains infinite-order and not implementable. A computable observer, therefore, must be one that uses the data sample; specifically, this observer should have an observer state that is forced to lie in the range of $\hat S$ (as a finite-dimensional subspace of $Z$). That is, the finite-order-approximate observer is written as 
\begin{equation}\label{eq:observer.finite-order}
    \hat{S} \dot{\hat \zeta} = \hat{\Delta}A_{\lambda,\tau}\hat{S} \hat{\zeta}+ \hat{L}(y - C\hat{S} \hat{\zeta}), \enspace \hat x= E\hat{S}\hat{\zeta}.
\end{equation}
Here $\hat\Delta$ is the projection operator from $Z$ onto $\ran \hat S$, and the term $A_{\lambda,\tau}\hat{S}$ is actually $\hat R_{\lambda,\tau}$. 
In \eqref{eq:observer.finite-order}, the observer state is $\hat{S}\hat{\zeta}(t)$, which belongs to $\ran \hat{S}$ and hence $\mr{span}\{\rg\kappa_{x_i}\}_{i=1}^n$. Therefore, the observer state $\hat\zeta(t)$ can in fact be represented by a vector $\ms{\xi}(t)\in \mR^n$ by $\hat\zeta(t) = \sum_{i=1}^n \xi_i(t)\rg\kappa_{x_i}$. 

To bound the difference between $\hat{z}_{\lambda,\tau}$ and $\hat{S}\hat\zeta$, we note that the essence of such discrepancy is that in \eqref{eq:observer.finite-order}, the flow of the observer state is restricted onto $\ran S$. 
\begin{definition}
    In \Cref{def:dissipative}, if $S$ is replaced by $\hat S$, we say that the system is \emph{dissipative by empirical sampling}.
\end{definition}

\begin{proposition}
    Suppose that $\{\rg\kappa_{x_i}\}_{i=1}^n$ are linearly independent in $Z$, and in addition, $\ker \hat P \supseteq (\mr{span}\{\rg\kappa_{x_i}\}_{i=1}^n)^\perp$ (i.e., the solution $\hat{P}$ for \eqref{eq:oALE.empirical} can be chosen as a data-based finite-rank operator; this is indeed true --- see next subsection). 
    Moreover, suppose that $(A_{\lambda,\tau}, C)$ is closed-loop dissipative by empirical sampling under observer gain $\hat L$. Then 
    \begin{equation*}
        \int_0^\infty |E(\hat{z}_{\lambda,\tau}(t)-\hat S\hat\zeta(t))|^2 \xD{t} < \infty,
    \end{equation*}
    and in particular, $E(\hat{z}_{\lambda,\tau}(t)-\hat S\hat\zeta(t)) \to 0$ as $t\to \infty$. 
\end{proposition}
\begin{proof}
     Given that $\{\rg\kappa_{x_i}\}_{i=1}^n$ are linearly independent, 
     \begin{equation*}
         \ran \hat{S} = \ran \bra{\frac{1}{n} \sum_{i=1}^n \rg\kappa_{x_i} \times \rg\kappa_{x_i}} \supseteq \mr{span}\{\rg\kappa_{x_i}\}_{i=1}^n. 
     \end{equation*} 
     Since $I-\hat\Delta$ is the projection operator onto $(\mr{span}\{\rg\kappa_{x_i}\}_{i=1}^n)^\perp$ and $(\mr{span}\{\rg\kappa_{x_i}\}_{i=1}^n)^\perp \subseteq \ker\hat{P}$ we have $\hat P(I-\hat\Delta) = 0$. Similarly, we ave $\hat{S}(I-\hat\Delta) = 0$. 
     Comparing \eqref{eq:observer.Yosida} and \eqref{eq:observer.finite-order}, we see that the error $\eta(t) = \hat{z}_{\lambda,\tau}(t)-\hat S\hat\zeta(t)$ satisfies the equation:
     \begin{equation*}
     \dot{\eta}(t) = (A_{\lambda,\tau}-\hat{L}C) \eta(t) + (I-\hat\Delta) (A_{\lambda,\tau}-\hat{L}C) \hat{S}\hat\zeta(t).
     \end{equation*}
     Suppose that $\eta(0)\in \ran\hat{S}$. Then by the closed-loop dissipativity hypothesis, from \Cref{rem:dissipativity}, we know that $\ip{\eta(t)}{\hat S^{-1} \eta(t)}$ remains bounded, i.e., $\eta(t)\in \ran \hat S^{1/2} = \ran \hat{S}$. 
     Consider the Lyapunov function $v(t) = \ip{\eta(t)}{\hat{P}\eta(t)}$. By differentiation, we obtain
     \begin{equation*}
     \begin{aligned}
        \dot{v}(t) =& \ip{\eta(t)}{\bra{ \hat{P}(A_{\lambda,\tau}-\hat{L}C) + (A_{\lambda,\tau}-\hat{L}C)^*\hat{P} }\eta(t)} \\
        &+ \ip{e}{\hat{P} (I-\hat\Delta) (A_{\lambda,\tau}-\hat{L}C) \hat{S}\hat\zeta(t)}, 
     \end{aligned}
     \end{equation*}
     where the second term on the right-hand side is zero, and the first term, due to the fact that $\eta(t)\in \ran \hat{S}$, is upper-bounded by $-k|E\eta(t)|^2$. Hence, 
     \begin{equation*}\int_0^\infty |E\eta(t)|^2 \xD{t} \leq \frac{1}{k}v(0) = \mr{const}.\end{equation*}
     The conclusion is proved. 
\end{proof}

Finally, we know that $z_{\lambda, \tau}(t)$ and $z(t)$ are governed by different infinitesimal generators: $A_{\lambda,\tau}$ and $A$. The bound on their error was given by \Cref{lem:Yosida}, which states that 
\begin{equation*}
    \|z_{\lambda,\tau}(t)-z(t)\|\lesssim (\lambda^{-1}+\tau^{-1})t\mr{e}^{\omega t}
\end{equation*}
where $\omega$ is the growth rate of the semigroup $\{T^t\}$ generated by $A$, as long as the initial condition $z(0)\in \mc D(A^2)$. 
Generally, we cannot hope for a better bound, because if $\{T^t\}$ is not contractive, the approximation by a Yosida approximation will naturally cause the long-term prediction error to accumulate. 
In the special case that $\omega<0$, the error ultimately vanishes; this can be the case, for example, under the conditions of \Cref{prop:stability-spectrum} where the system is homeomorphically transformed to a stable Jacobian dynamics. 
If $Z_0$ and $Z_1$ are both invariant (so that $\{T^t\}$ can be block diagonalized), we can also argue that, if only the original $d$-dimensional states are concerned, the $\omega$ to be considered is the one for $\{T^t|_{Z_1}\}$. 

\begin{proposition}\label{th:observer.performance.2}
    Suppose that $\|T^t\|\lesssim \mr{e}^{\omega t}$ for all $t\in \mR_+$ and some $\omega<0$. Then 
    the observer synthesized from $(\hat{P}, \hat{Y})$ guarantees:
    \begin{equation*}\int_0^\infty |E(z_{\lambda,\tau}(t) - z(t))|^2 \xD{t} < \infty, \end{equation*}
    and hence $E(z_{\lambda,\tau}(t) - z(t)) \to 0$ as $t\to \infty$. 
\end{proposition}
\begin{proof}
     By the above discussions, we have $|Ez_{\lambda,\tau}(t)-z(t)|\lesssim (\lambda^{-1}+\tau^{-1})t\mr{e}^{\omega t}$ and hence
     \begin{equation*}\int_0^\infty |E(z_{\lambda,\tau}(t)-z(t))|^2 \xD{t} \lesssim \bra{\frac{1}{\lambda}+ \frac{1}{\tau}}^2 \int_0^\infty t^2 \mr{e}^{\omega t}\xD{t}.\end{equation*}
     When $\omega<0$, the right-hand side is finite. 
\end{proof}

To summarize the above three conclusions, we have the following theorem on the state observation performance. 
\begin{theorem}
    Suppose that the empirical operator Lyapunov inequality \eqref{eq:oALE.empirical} has a solution $(\hat P, \hat Y)$ and hence $\hat L = \hat P^{-1}\hat Y$. Assume that $\{\rg\kappa_{x_i}\}_{i=1}^n$ are linear independent, and the system $(A_{\lambda,\tau}, C)$ is dissipative by sampling and empirical sampling under observer gain $\hat L$. 
    Then, with probably $1-\delta$, if the sample $\{x_i\}_{i=1}^{n}$ with $n>n_{\ve,\delta}$ is drawn i.i.d. from $\nu$, the observed state $\hat x(t)$ from the finite-order observer \eqref{eq:observer.finite-order}, compared to the Yosida-regularized system \eqref{eq:ACE.Yosida}, satisfies a time-averaged squared error in the sense of:
    \begin{equation*}
        \frac{1}{\theta}\int_0^\theta |\hat{x}(t) - Ez_{\lambda,\tau}(t)|^2 \xD{t} \lesssim \frac{\ve}{k} \bra{1+\|A_{\lambda,\tau}\|} + \frac{1}{\theta}   
    \end{equation*} 
    for all $\theta>0$, where $z_{\lambda,\tau}$ evolves under \eqref{eq:ACE.Yosida}.
    If moreover $\|T^t\|\lesssim \mr{e}^{\omega t}$ for all $t\in \mR_+$ for some $\omega<0$, then the above bound holds for the true system \eqref{eq:ACE}, i.e., 
    \begin{equation*}
        \frac{1}{\theta}\int_0^\theta |\hat{x}(t) - x(t)|^2 \xD{t} \lesssim \frac{\ve}{k} \bra{1+\|A_{\lambda,\tau}\|} + \frac{1}{\theta}.  
    \end{equation*} 
\end{theorem}

\begin{remark}
	While it is \emph{sufficient} that the exponential stability $\{T_t\}_{t\geq 0}$ guarantees $E(z_{\lambda,\tau}(t) - z(t)) \to 0$, such a condition is strong and unnecessary. In later numerical examples on limit cycle and chaotic systems where the spectrum contains an infinite amount of points on the imaginary axis, the difference between the Yosida-approximate state $Ez_{\lambda,\tau}(t)$ and the true state $Ez(t)$ does not diverge; instead, such a difference appears to be bounded and dependent on the choice of Yosida regularization parameters $(\lambda, \tau)$, although such a bound cannot be derived readily from the semigroup theory. 
\end{remark}

\subsection{Elementary Expressions and Solution} \label{subsec:elementary}
Following the discussions in \Cref{subsec:compactified}, it becomes obvious that the computational tasks for the observer synthesis boils down on finite-order matrices. 
That is, we can let $P$ and $Y$ have the following matrix representations, where the infinite-dimensional $Z$ is reduced to the finite-dimensional span of the data-based kernel functions. 
\begin{equation*}
    \hat P = \sum_{i,j=1}^n \ms{P}_{ij} \rg\kappa_{x_i}\times \rg\kappa_{x_j} , \ \ \hat Y = \sum_{i=1}^n \sum_{j=1}^m \ms{Y}_{ij} \rg\kappa_{x_i} \times \ms{e}_j.
\end{equation*}  
Also, we note that $C = \sum_{i=1}^m \ms{e}_i\times h_i$ and $E = \sum_{i=1}^d \ms{e}_i\times e_i$. 
For computation, we introduce the kernel matrices:
\begin{equation}\label{eq:kernel.matrices}
    \begin{aligned}\ms{G} &= \Bra{\ip{\rg\kappa_{x_i}}{\rg\kappa_{x_j}}}_{i,j=1}^n = \Bra{\rg\kappa(x_i, x_j)}_{i,j=1}^n, \\
    \ms{G}' &= \Bra{\ip{\rg\kappa'_{x_i}}{\rg\kappa_{x_j}}}_{i,j=1}^n, \end{aligned}
\end{equation}
in which, for any two points $x_i$ and $x_j$, 
\begin{equation*}
\ip{\rg\kappa'_{x_i}}{\rg\kappa_{x_j}} = \lambda^2\int_0^\tau \mr{e}^{-\lambda t}\rg\kappa(x_i(t), x_j)\, \xD{t} - \lambda \rg\kappa(x_i, x_j). 
\end{equation*}
We also write the data matrices:
\begin{equation}\label{eq:data.matrices}
\begin{aligned}
    \ms{X} = \begin{bmatrix} x_1&\cdots & x_n \end{bmatrix}, \enspace 
    \ms{H} = \Bra{h_i(x_j)}_{1\leq i\leq m, 1\leq j\leq n}, 
\end{aligned}
\end{equation}

\par Then it can be verified that the first term in the empirical Lyapunov inequality \eqref{eq:oALE.empirical} is re-expressed by
\begin{equation*} 
\hat S \hat P \hat R_{\lambda,\tau} = \frac{1}{n^2} \sum_{i,j=1}^n \bra{ \ms{G}\ms{P}\ms{G}'^\top }_{ij} \rg\kappa_{x_i}\times \rg\kappa_{x_j}. 
\end{equation*}
The third term in \eqref{eq:oALE.empirical} becomes
\begin{equation*}
\hat{S}\hat{Y}C\hat{S} = \frac{1}{n^2} \sum_{i,j=1}^n \bra{ \ms{G} \ms{Y}\ms{H}}_{ij} \rg\kappa_{x_i}\times \rg\kappa_{x_j},
\end{equation*}
and the fifth term is written as 
\begin{equation*}
k\hat{S}E^*ES = \frac{k}{n^2}\sum_{i,j=1}^n \bra{\ms X^\top \ms X}_{ij} \rg\kappa_{x_i}\times \rg\kappa_{x_j}.
\end{equation*}
Now the empirical operator Lyapunov equality \eqref{eq:oALE.empirical} is fully converted to a linear matrix inequality about $(\ms{P}, \ms{Y})$:
\begin{equation}\label{eq:elementary.oALE}
\ms{G}\ms{P}\ms{G}'^\top + \ms{G}'\ms{P}\ms{G}^\top - \ms{G}\ms{Y}\ms{H} - \ms{H}^\top\ms{Y}^\top\ms{G}^\top + k\ms{X}^\top\ms{X} \preceq 0.
\end{equation}
To impose the positive definiteness constraint \eqref{eq:P.psd} on $P|_{Z_1}$, we know that it suffices to require $\sum_{k,\ell=1}^d a_ka_\ell \ip{e_k}{Pe_\ell} \geq \sum_{k=1}^d a_k^2$ to hold for all $a\in \mR^d$. This reduces to
\begin{equation}\label{eq:elementary.P.psd}
    \ms{X}\ms{P}\ms{X}^\top - \ms{I} \succeq 0, \enspace \ms{P}\succeq 0. 
\end{equation}
Therefore, the solution boils down to a pair of matrices $\ms{P}\in \mR^{n\times n}$ and $\ms{Y}\in \mR^{d\times m}$ such that \eqref{eq:elementary.oALE} and \eqref{eq:elementary.P.psd} are satisfied. 
To facilitate convex optimization solvers, it may be beneficial to set an objective function $J(\ms{P})$ to be minimized, e.g., $\operatorname{tr}(\ms{P})$, in the semidefinite programming problem:
\begin{equation}\label{eq:elementary.optimization}
    \min_{\ms{P}\in\mR^{n\times n}, \ms{Y}\in\mR^{n\times m}} \operatorname{tr}(\ms{P}) \ \ \text{s.t. }  \eqref{eq:elementary.oALE}, \eqref{eq:elementary.P.psd}. 
\end{equation}

\par Once solved, the expression of $L$ can be obtained:
\begin{equation}\label{eq:elementary.L}
    \begin{aligned}&L = \sum_{i=1}^n\sum_{j=1}^m \ms{L}_{ij} \rg\kappa_{x_i}\times \ms{e}_j, \\ &\text{where } \ms{L}=(\ms{GPG})^{-1}\ms{G} \ms{Y} = (\ms{P}\ms{G})^{-1}\ms{Y}. \end{aligned}
\end{equation}
Finally, we plug the above expression into the finite-order-approximate observer \eqref{eq:observer.finite-order}. With the finite-dimensional expression $\hat\zeta(t) = \sum_{i=1}^n \xi_i(t)\rg\kappa_{x_i}$ therein, it suffices that 
\begin{equation*}\begin{aligned}\ip{\rg\kappa_{x_i}}{\hat S \dot{\hat{\zeta}}(t)} &= \ip{\rg\kappa_{x_i}}{\hat R_{\lambda,\tau}\hat\zeta + \hat L(y-C\hat S\hat\zeta)}, \\ &\quad \forall i=1,\cdots,n.\end{aligned}\end{equation*}
After elementary calculation, the observer is simplified to:
\begin{equation*}
\begin{aligned}
    \frac{1}{n}\ms G^2\dot\xi(t) &= \frac{1}{n}\ms{G}'^\top \ms{G}\xi(t) + \underbrace{ \ms{G}\ms{L} }_{\ms{P}^{-1}\ms{Y}} \bra{y(t) - \frac{1}{n}\ms{HG}\xi(t)}, \\
    \hat{x}(t) &= \frac{1}{n}\ms{XG}\xi(t).
\end{aligned} 
\end{equation*}
By letting $\ms{z}(t) = \frac{1}{n}\ms{G}\xi(t)$, we reach at 
\begin{equation}\label{eq:observer.final}
\begin{aligned}
    \ms{G}\dot{\ms{z}}(t) = \ms{G}'^\top \ms{z}(t) + \ms{P}^{-1}\ms{Y} \bra{ y(t)-\ms{H}\ms{z}(t) }, \enspace
    \hat{x}(t)= \ms{X}\ms{z}(t). 
\end{aligned}
\end{equation}

\begin{remark}
    The finally obtained finite-order observer has an intuitive interpretation. The $n$-dimensional vector $\ms{z}$ can be thought of as the vector of weights of each sample point in the state estimate, so that $\ms{X}\ms{z}$ is the returned $\hat x$. The innovation quantity that drives the evolution is $y(t)-\ms{H}\ms{z}(t)$, where $\hat y(t)=\ms{H}\ms{z}(t)$ is the weighted average of predicted output contributed by the sample points. 
    Since $\ms z=\frac{1}{n}\ms G\xi$, we have $\hat S\hat\zeta = \frac{1}{n}\sum_{i=1}^n \hat\zeta(x_i)\rg\kappa_{x_i} = \sum_{i=1}^n \ms z_i\rg\kappa_{x_i}$, i.e., $\ms z$ is the coefficient vector of the observer state in the kernel basis, and $\hat x = E\hat S\hat\zeta = \sum_{i} \ms z_i E\rg\kappa_{x_i} = \ms X\ms z$, where $E\rg\kappa_{x_i}=x_i$ follows from the reproducing property. 
    The appearance of kernel matrix $\ms{G}$ reflects the topology of RKHS, and $\ms{G}'^\top$ is a Yosida approximation of the system's evolution when lifted into the RKHS. The Luenberger observer gain $\ms{P}^{-1}\ms{Y}$ resembles that would be expected from a finite-dimensional linear system. 
\end{remark}

\section{Numerical Examples}\label{sec:numerical}
In this section, three systems with different behaviors on their state spaces are studied. The first system is an asymptotically stable one, the second has a limit cycle, while the third exhibits a chaotic behavior. 
As should be emphasized, the purpose is to demonstrate that the proposed Koopman--Luenberger provides a \emph{uniform framework for observer synthesis from data despite the severe nonlinearity} in a computationally efficient way, instead of outperforming the existing system-specific, model-based, or even fine-tuned state observers. 
Naturally, if a precise model on the state space is known, if a particular type of state-space geometry is considered, or if the observer can be fine-tuned with computationally expensive methods (e.g., deep learning), then it is always possible to reach a comparable performance. 

For the reader's information, we show the performance of Koopman--Luenberger observer (by their integrated squared errors $\text{ISE} = \int_0^{t_{\max}} |\hat{x}(t)-x(t)|^2 \xD{t}$ over the simulation interval $[0, t_{\max}]$), in comparison to:
\begin{enumerate}[label=(\roman*)]
    \item a linear model-based Luenberger observer, where we linearize the system at the origin to obtain $(A_\mr{lin}, C_\mr{lin})$, solve the linear matrix inequality $P_{\mathrm{lin}} A_{\mathrm{lin}} + A_{\mathrm{lin}}^\top P_{\mathrm{lin}} - Y_{\mathrm{lin}} C_\mr{lin} - C_\mr{lin}^\top Y_{\mathrm{lin}}^\top + k I \preceq 0$ using the same $k$ constant with the objective of minimizing $\operatorname{tr}(P_{\mr{lin}})$, and let $L_{\mr{lin}} = P_{\mr{lin}}^{-1}Y_{\mr{lin}}$ be the gain; 
    \item a nonlinear model-based Luenberger observer $\dot{\hat{x}} = f(\hat{x}) + L_{\mr{lin}}(y-h(\hat{x}))$, assuming that the ground-truth nonlinear model is completely known and using the linear Luenberger gain;  
    \item a Kazantzis--Kravaris/Luenberger observer \cite{kazantzis1998nonlinear}: 
    \begin{equation*}\dot{z}(t) = A_{\mr{KKL}}z(t) + B_{\mr{KKL}} y(t), \enspace \hat{x}(t) = \mr{NN}(z(t)), \end{equation*}
    where $\mr{NN}$ refers to a nonlinear mapping trained as a neural network (using the same trajectory data as used for Koopman--Luenberger observer). We used $32$ neurons in $2$ hidden layers to ensure a good approximation capacity within a comparable training time. In theory \cite{andrieu2006existence, pachy2024existence}, the mapping from $x$ to $z$ is an injection and hence invertible. 
\end{enumerate} 
The results are summarized in \Cref{tab:observer_comparison}.\footnote{ 
The codes to generate the results are made available at the author's GitHub repository: \url{https://github.com/WentaoTang-Pack/KoopmanLuenbergerObserver}, and are run in Python 3.13.2 on a laptop with an Apple M4 Max chip. The semidefinite programming solver used for \eqref{eq:elementary.optimization} is MOSEK 11.2.4.}

\begin{table}[!t]
\centering
\caption{Comparison of state observation performance and training time.}
\label{tab:observer_comparison}
\setlength{\tabcolsep}{6pt}
\renewcommand{\arraystretch}{1.1}
\begin{tabular}{llrr}
\toprule
\textbf{Example} & \textbf{Observer} & \textbf{ISE} & \textbf{Time (s)} \\
\midrule
\multirow{4}{*}{%
  \shortstack[l]{%
    Ex.~1
  }%
}
  & Linear model-based             & $0.7348$ & $-$  \\
  & Nonlinear model-based          & $0.5839$ & $-$  \\
  & Neural-network KKL             & $0.5015$ & $5.0$ \\
  & Koopman--Luenberger             & $0.4619$ & $0.25$ \\
\midrule
\multirow{4}{*}{%
  \shortstack[l]{%
    Ex.~2
  }%
}
  & Linear model-based             & $1.0493$ & $-$   \\
  & Nonlinear model-based          & $0.1577$ & $-$   \\
  & Neural-network KKL             & $0.6084$ & $5.2$  \\
  & Koopman--Luenberger             & $0.4233$ & $0.25$  \\
\midrule
\multirow{4}{*}{%
  \shortstack[l]{%
    Ex.~3
  }%
}
  & Linear model-based             & $21580$        & $-$    \\
  & Nonlinear model-based          & $58$            & $-$    \\
  & Neural-network KKL             & $1531$         & $7.8$  \\
  & Koopman--Luenberger             & $277$         & $1.7$  \\
\bottomrule
\end{tabular}
\end{table}

\subsection{An Asymptotically Stable System}\label{subsec:example1}
The model considered here is an isothermic continuously stirred tank reactor, assumed to be perfectly mixed, with a fractional reaction rate depending on the concentration of the reactant ($x_1$) and independent from the product concentration ($x_2$). The state coordinate is shifted so that the origin is an equilibrium point. The equations are:
\begin{equation*}\dot x_1 = \frac{3-x_1}{4} - \frac{9(1+x_1)}{4(3+2x_1)}, \enspace \dot x_2 = -\frac{3(1+x_2)}{4} + \frac{9(1+x_1)}{4(3+2x_1)} \end{equation*}
where only $y=x_1$ is measured. The state space is chosen as $\mbb X = [-1, 1]\times [-1,1]$. From its Jacobian at the origin, we know that the system is asymptotically stable. 

\par For the linear--radial kernel, we use a smoothness index $s=2.5$ and a length scale $\ell=1.0$ that matches the size of the sampling region. It was found that the tuning of $s$ and $\ell$ does not have a significant impact on the observer performance (and for brevity, these results are omitted). 
For finite-time Yosida approximation, we used $\lambda=5$ and $\tau=3$. The effect of these two parameters is shown in \Cref{fig:tuning.Yosida}, where $A_{\mr{cl}}$ refers to the closed-loop matrix $\ms{G}^{-1}\bra{\ms{G}'^\top - \ms{P}^{-1}\ms{Y}\ms{H}}$ of the observer \eqref{eq:observer.final}. 
Roughly, when $\lambda\tau>30$, the ill-conditioned kernel matrices causes the observer simulation to become numerically unstable and hence the performance cannot be evaluated (marked by a cross-out symbol); this is the region of parameters that should be avoided, and within the safe region, the performance is also robust to parameter choices. 
While $\max \Re(A_{\mr{cl}})$ (the maximum real part of the eigenvalues) is not directly related to the state observation performance (since the state is $d$-dimensional while the observer state $\ms{z}$ is in an $n$-dimensional space), it is desirable to have an empirically Schur matrix $A_{\mr{cl}}$. Hence, the choice of $(\lambda,\tau)=(5,3)$ is fair. 
\begin{figure}[t!]
    \centering
    \subfloat[Yosida approximation parameters $(\lambda,\tau)$.\label{fig:tuning.Yosida}]{\includegraphics[width=0.9\columnwidth]{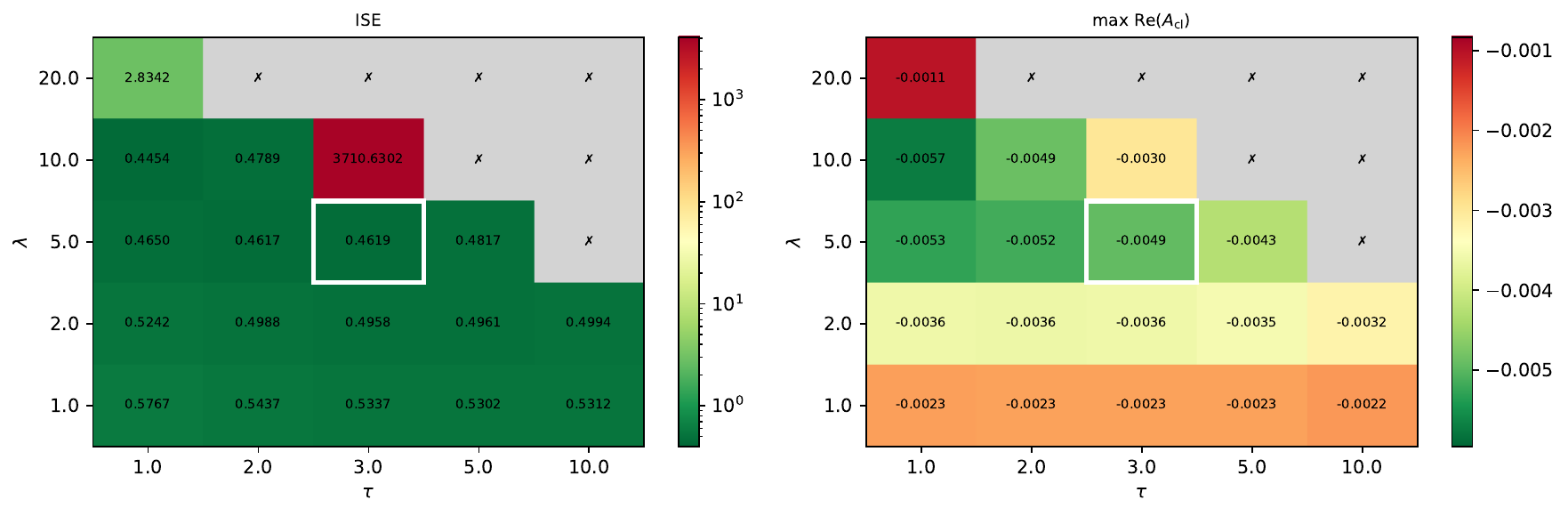}} \\
    \subfloat[Forcing parameter $k$.\label{fig:tuning.forcing}]{\includegraphics[width=0.48\columnwidth]{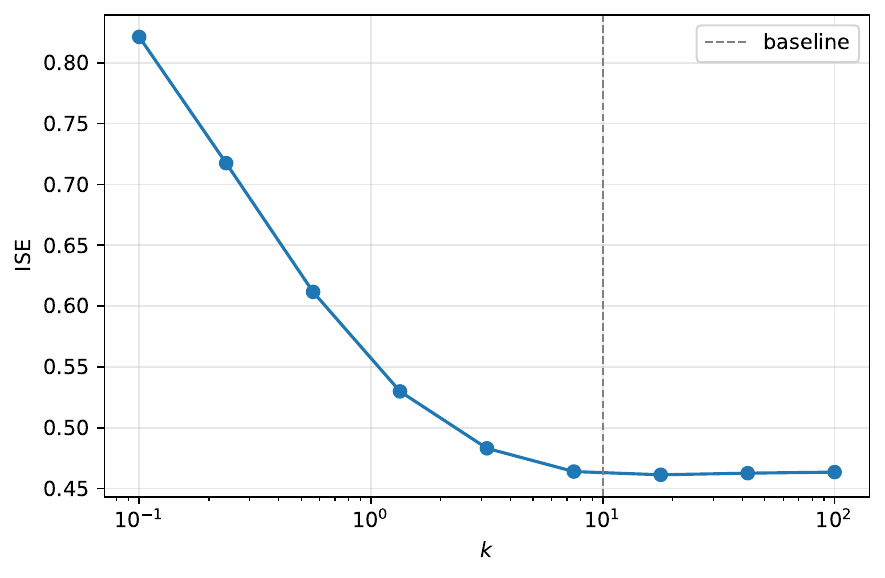}}\hfill
    \subfloat[Sample size $n$.\label{fig:tuning.sample-size}]{\includegraphics[width=0.48\columnwidth]{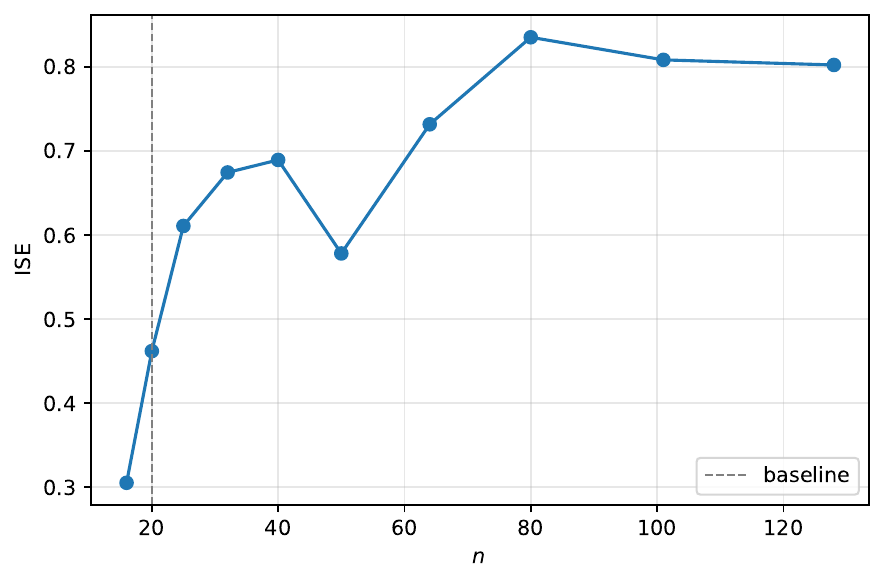}}
    \caption{Hyperparameter tuning in observer synthesis in Example 1.}\label{fig:tuning}
\end{figure}

\par \Cref{fig:tuning.forcing} shows the dependence of performance metric (ISE) on the forcing parameter $k$ in the synthesis formulation \eqref{eq:observer.Yosida}. Since larger $k$ represents the need for a more aggressive compensation by the innovation quantity ($y-C\hat z$), the ISE curve is a decreasing function of $k$ as expected. At $k=10$, the ISE has reached a plateau. 
Interestingly, as shown by \Cref{fig:tuning.sample-size}, the observer performance tends to be better when the sample size is even smaller (except when $n$ is extremely low). In this study, we choose $n=20$ as the baseline. This owes to the conditioning issue -- at larger $n$, the kernel matrix that arises from the empirical sampling operator $\hat S$ becomes more ill-conditioned, and hence for the computational solution of the LMI \eqref{eq:elementary.oALE}, one must replace $\ms{G}$ by a Lavrentiev-regularized $\ms{G}+\epsilon_1 \ms{I}$, enforce $\ms{P}\succeq \epsilon_2 \ms{I}$, and tighten the right-hand side of \eqref{eq:elementary.oALE} by $-\epsilon_3\ms{I}$ for some small $\epsilon_{1,2,3}>0$. 
These treatments inevitably introduces sub-optimality to the resulting observer. 

\par Finally, the comparison of four different types of observers are shown in the simulated trajectories in \Cref{fig:comparison-1}, where all observer states are initialized at the origin. 
The two model-based observers achieve faster decay in observation error in the long run, due to their knowledge of the linearized model near the origin, while at the incipient stage, the neural KKL observer and Koopman--Luenberger observer track the true states better. 
The proposed Koopman--Luenberger observer, within more than one order-of-magnitude lower training time budget than the neural network approach, achieves a small error and a better performance in terms of ISE.
As anticipated in \Cref{rem:spectral-gap}, no exponential error decay is expected. Nevertheless, asymptotic decay is obtained, and the error decreases to a practically sufficiently small value rapidly. 
\begin{figure}
    \centering
    \includegraphics[width=\columnwidth]{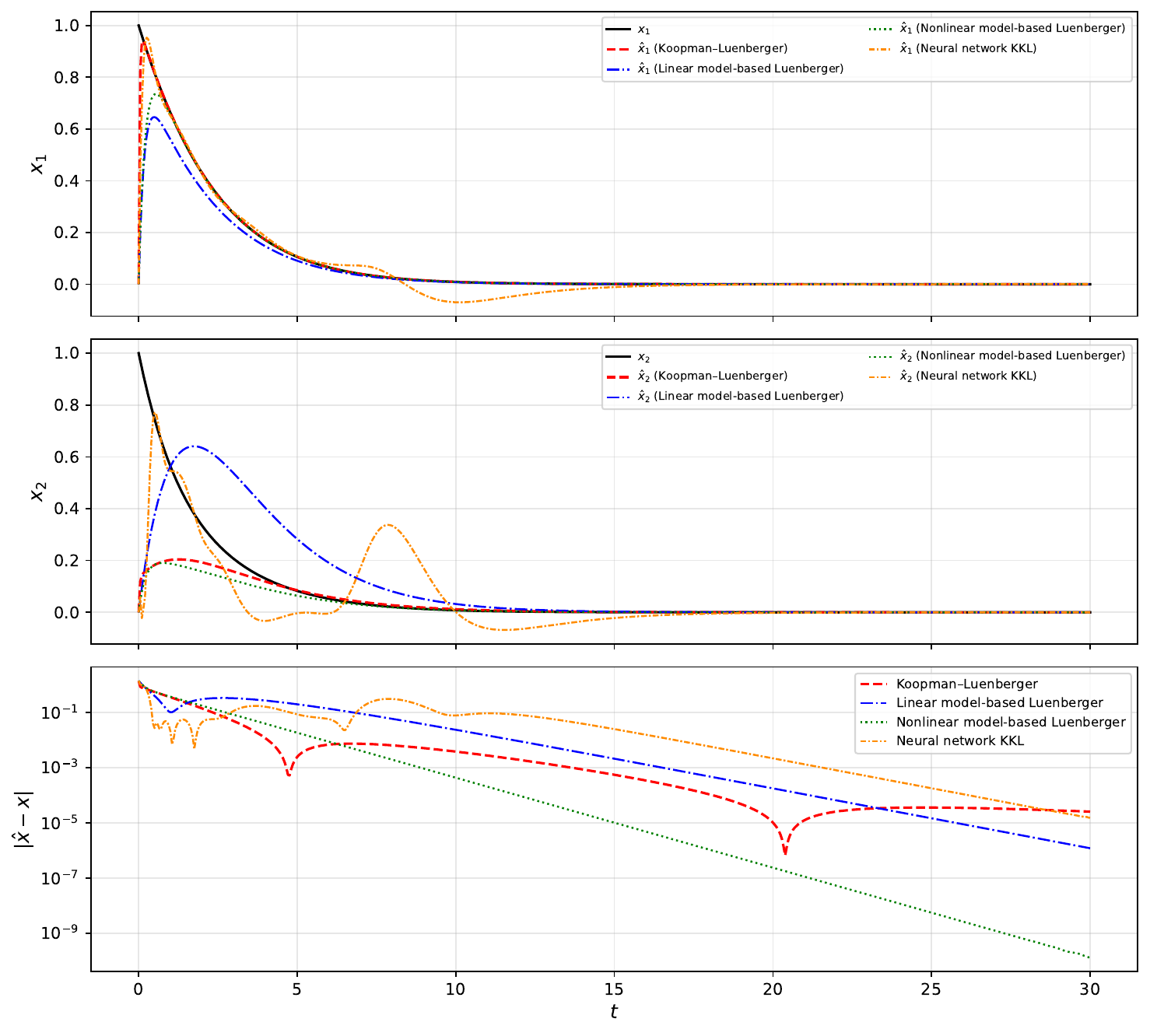}
    \caption{Simulated trajectories of observers for Example 1.}
    \label{fig:comparison-1}
\end{figure}

\subsection{A Limit Cycle System}
For the van der Pol oscillator, the origin is unstable and a limit cycle behavior can be seen. We use the model:
\begin{equation*}\dot x_1 = x_2, \enspace \dot x_2 = 0.5(1-9x_1^2)x_2 -x_1,\end{equation*}
and assume that only $y=x_1$ is measured. 
All hyperparameters used remain identical to the previous example, and the effects of tuning are found to have similar patterns as before. Hence the tuning plots are not presented. We again use a sample size of $n=20$. 
\begin{figure}[!t]
    \centering
    \includegraphics[width=\columnwidth]{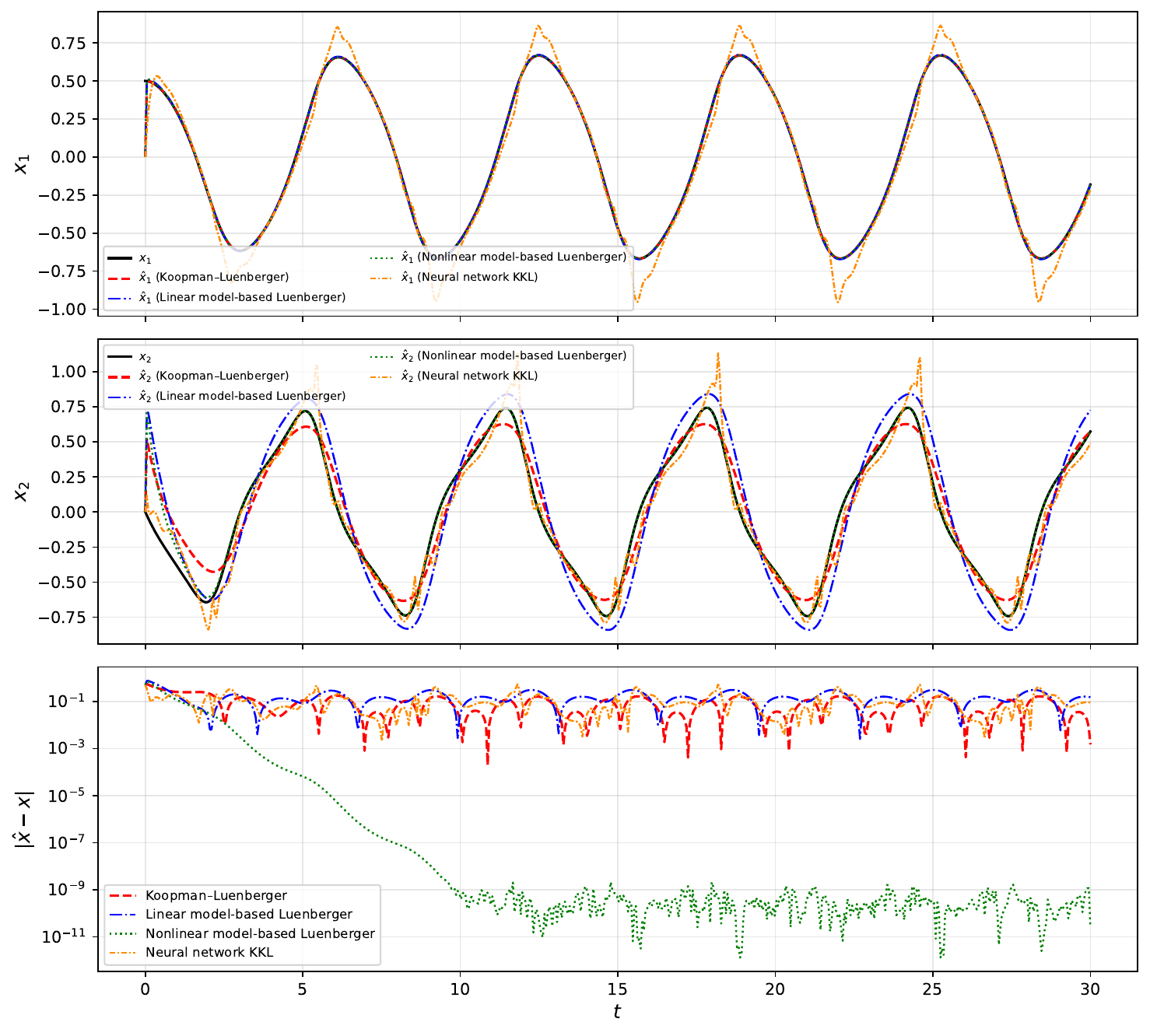}
    \caption{Simulated trajectories of observers for Example 2.}
    \label{fig:comparison-2}
\end{figure}

The comparison plots for the four different observers are provided in \Cref{fig:comparison-2}. It is still noticed that the proposed Koopman--Luenberger observer achieves a performance that is better than a linearized model-based Luenberger observer, and also better than a neural KKL observer trained using the same amount of data (after $5000$ epochs). The time needed to train (optimize) a Koopman--Luenberger observer is much shorter than that for neural networks. 
Only under a perfect nonlinear model, the observer performance surpassed the Koopman--Luenberger observer.

\subsection{A Chaotic System}
We examine the classical Lorenz system, where the nonlinearity is considered to be more severe than the previous two examples:
\begin{equation*}\begin{aligned}&\dot x_1 = 10(x_2-x_1), \ \ \dot x_2 = x_1(28-x_3)-x_2, \\ &\dot x_3 = x_1x_2 - (8/3)x_3,\end{aligned}\end{equation*}
Among the three states, $y_1=x_1$ and $y_2=x_3$ are considered to be measurable. 
In this example, if the data points are sampled in a hypercube centered at the origin, then such a hypercube must be very large to include the Lorenz attractor and hence the points are sparse and tend to be far away from the attractor. 
Therefore, instead of directly generating the sample in a hypercube, a burn-in trick is used, i.e., starting from these random points, a simulation of $10$ time units is first carried out to settle the state close enough to the attractor, and the final states are adopted as the sample. 
 
Within the same sample size and the same amount of training epochs as the previous two examples, it takes slightly more time to train the neural KKL observer due to the higher state dimension, while the computational time for solving the semidefinite programming problem for Koopman--Luenberger observer also becomes higher. 
The two training processes are completed in a similar amount of time ($10.1$ versus $12.8$ seconds), and in this case the Koopman--Luenberger observer achieves a better performance in this example. The performance improvement over the neural approach is $5.53$-fold (gauged by ISE) or $2.35$-fold (by root mean squared error).  
\begin{figure}[!t]
    \centering
    \includegraphics[width=\columnwidth]{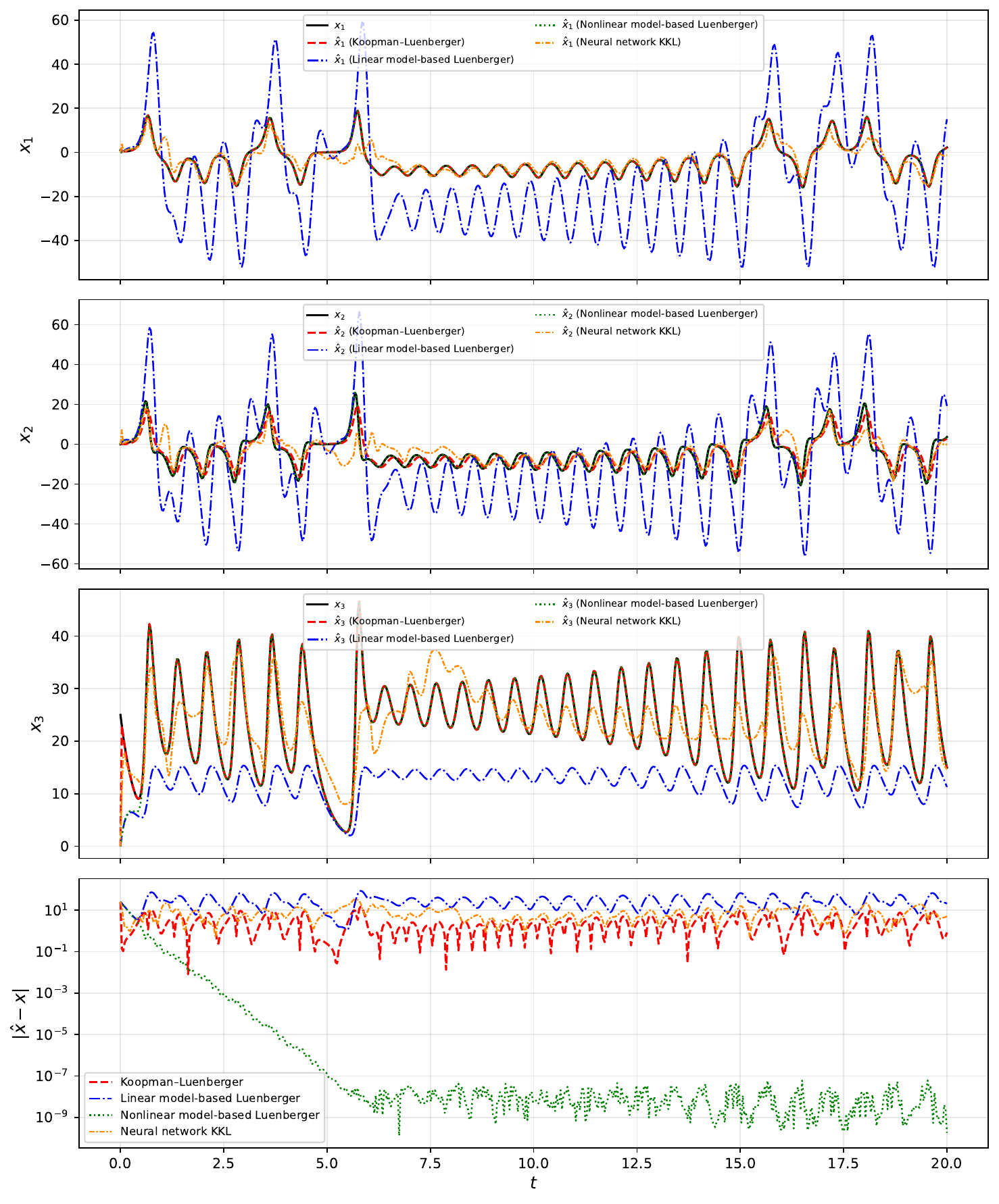}
    \caption{Simulated trajectories of observers for Example 3.}
    \label{fig:comparison-3}
\end{figure}

\section{Discussions and Conclusions}\label{sec:conclusions}
This paper aims to give a systematic treatment to the problem of designing a Luenberger-type state observer for nonlinear systems, represented as infinite-dimensional linear systems via Koopman operator semigroup and its adjoint.  
The Koopman operator semigroup is defined on an RKHS whose reproducing kernel is a linear--radial kernel \cite{tang-ye2025koopman}. 
The Koopman-version concepts of exact observability, approximate observability, detectability, strong detectability, and exponentially detectability are defined and the theoretical relations are studied. 
An operator Lyapunov inequality is formulated to specify a Koopman--Luenberger observer, and towards data-driven solution, the inequality is approximated in a Yosida form and further projected to a finite-order form. 
Theoretical analysis justifies these approximations with a guarantee (in a probabilistic sense) of bounded state observation error in time average, and numerical examples show the applicability to nonlinear systems in wide generality with outstanding computational efficiency.

\par It may be of interest to remark the conceptual connections and distinctions between Kazantzis--Kravaris/Luenberger observer \cite{andrieu2006existence} and the Koopman--Luenberger observer contrived in this paper. 
In both categories, the observer state is made to have a linear time-invariant (LTI) dynamics, driven by the measured output variables. For this construction, the state $x$ has to be transformed nonlinearly and then restored from the transformed appearance $z$. The difference lies in that in KKL, the observer uses a finite-dimensional $z$, whose dimension is higher than or can be equal to that of $x$ \cite{pachy2024existence}, and the inversion from $z$ to $x$ must be a nonlinear mapping; in contrast, in Koopman--Luenberger, with an RKHS construction, the retrieval of $x$ from an infinite-dimensional $z$ is completely linear. 
Such a construction difference affects how the solution takes place. In KKL, the solution of the nonlinear pseudoinverse mapping is the key difficulty, typically addressed by polynomial series (i.e., using a linear combination of monomial features) \cite{kazantzis1998nonlinear} or by neural networks that universally approximates nonlinear mappings via nonlinear parameterization \cite{niazi2022learning, tang2024synthesis, woelk2026neural}; in Koopman--Luenberger, the solution of a Luenberger gain boils down to a convex optimization problem. 
\emph{The conceptual relation between KKL and Koopman--Luenberger observers is analogous to that between nonlinear regression (via hand-crafted features or nonlinear parameterization) and kernel-based regression techniques in statistical learning}. The former is restrictive, dependent on the feature selection, and gives rise to nonconvexity, but may be accurately trainable given intensive computational effort. The latter is general, convex, essentially nonparametric, and therefore can be more subject to overfitting, which explains the need of extra dissipativity conditions (\Cref{def:dissipative}). 

\par The readers may have noted that in the differential-geometric theory of nonlinear systems, the observability and controllability concepts have long been formulated in the language of Lie algebra \cite{isidori1985nonlinear, jurdjevic1997geometric}. It may be worth studying the connection of Koopman observability concepts with the geometric ones.  
Another important question is the observation and control for systems with inputs, especially the bilinearity issue that arises therefrom.

\section*{Acknowledgment}
The author acknowledges the use of Claude Code (Anthropic), powered by Claude Sonnet 4.6, for assistance with code development in this work. All AI-generated code was reviewed, tested, and validated by the author, who takes full responsibility for its correctness.

\bibliographystyle{IEEEtran}
\bibliography{koopman_tac}


\end{document}